\documentclass[a4paper]{amsart}
\usepackage{amsmath}
\usepackage{cases}
\usepackage{amsfonts}
\usepackage[colorlinks,linkcolor=blue,citecolor=blue]{hyperref}
\usepackage{latexsym, amssymb, amsmath, amsthm, bbm}
\usepackage[all]{xy}
\usepackage{pgfplots}
\usepackage [latin1]{inputenc}

\usepackage{anysize}\marginsize{30mm}{30mm}{30mm}{30mm}

\def \k{\mathbbm{k}}

\numberwithin{equation}{section}
\newtheorem{theorem}{Theorem}[section]
\newtheorem{lemma}[theorem]{Lemma}
\newtheorem{proposition}[theorem]{Proposition}

\newtheorem{definition}[theorem]{Definition}

\newtheorem{remark}[theorem]{Remark}

\begin{document}
\title{Finite dimensional Nichols algebras over $H_{c: \sigma_{0}}$ of Kashina}
\thanks{$^\dag$Supported by NSFC 12171230.}

\author[Miantao Liu, Gongxiang Liu, Kun Zhou]{Miantao Liu, Gongxiang Liu, Kun Zhou}
\address{Department of Mathematics, Nanjing University, Nanjing 210093, China}
\email{dg1921007@smail.nju.edu.cn}
\email{gxliu@nju.edu.cn}
\address{Yanqi Lake Beijing Institute of Mathematical Sciences and Applications, Beijing 101408, China}
\email{kzhou@bimsa.cn}


\date{}
\maketitle
\begin{abstract}
Let $H$ be the Hopf algebra $H_{c: \sigma_{0}}$ of Kashina [J. Algebra, 232(2000),pp.617-663]. We give all simple Yetter-Drinfel'd modules $V$ over $H$, then classify all finite-dimensional Nichols algebras of $V$. The finite dimensional Nichols algebras of diagonal type are either  $A_{1}, A_{2}$ or quantum planes, and non-diagonal type ones are $8$ or $16$ dimensional.
\end{abstract}

\section{Introduction}
 Nichols algebras appear naturally in the classifications of pointed Hopf algebras~\cite{Andru10} and have many applications in other fields such as conformal field theory \cite{SA}, Lie algebra \cite{ARF} etc. Among of them, Nichols algebras of diagonal type have attracted a lot of attention and were understood quite well now \cite{Andru19,Andru17,Andru172,Andru21,Hecken06, Hecken07, Hecken09,Hecken18}. Furthermore, authors are also interested in finite-dimensional Nichols algebras of other types. Meanwhile, Drinfel'd double of Hopf algebras was introduced by Drinfel'd in order to study quasitriangular Hopf algebras. The representation theory of Drinfel'd double of Hopf algebra  has been studied for many years. The relation between the category of left-right Yetter-Drinfel'd modules over $H$ and the category of left modules over the Drinfel'd double $\mathcal{D}(H)$ has been considered in the paper~\cite{Maji91}. If $H$ is a finite-dimensional Hopf algebra, then these two categories are equivalent.

Let $H$ be the 16 dimensional Hopf algebra $H_{c: \sigma_{0}}$ of Kashina \cite{Kashina}. In this paper, we start with investigating all simple Yetter-Drinfel'd modules $V$ over $H$. We proved that there are 88 non-isomorphic simple left $\mathcal{D}(H)$-modules, and 32 of them are one-dimensional and the other 56 ones are two-dimensional. Then we classify all finite-dimensional Nichols algebras of simple Yetter-Drinfel'd modules $V$ over $H$. We find once again the Nichols algebras of non-diagonal type which were studied by Andruskiewitsch~\cite{AG} and non-diagonal type ones are $8$, $16$ or infinite dimensional (see Lemma~\ref{theorem5.4}).

The paper is organized as follows. In Section 2, we review some definitions, notations and results related to Yetter-Drinfel'd modules, Drinfel'd double, Nichols algebra and the Hopf algebra $H$. In Section 3, we describe the Drinfel'd double $\mathcal{D}(H)$ of $H$. We devote Section 4 and 5 to giving our main results (see Theorem \ref{theorem4.7} and Theorem \ref{theorem5.6}).
Throughout the paper, we work over an algebraically closed field $\Bbbk$ of characteristic zero. All Hopf algebras in this paper are finite dimensional. Our references for the theory of Hopf algebras are \cite{Mon1993,Rad12}. For a Hopf algebra $H$, the antipode of $H$ will be denoted by $\mathcal{S}$.

\section{Preliminaries}
In this section, we recall some well-known facts about Yetter-Drinfel'd modules, Drinfel'd double and Nichols algebra. We also describe the Hopf algebra $H_{c: \sigma_{0}}$ as an abelian extension.
\subsection{Some concepts}
 All concepts are standard, and one can find the detailed analysis about them in many references (see \cite{Andruskiewitsch,Rad12}).
\begin{definition} Let $V$ be a vector space over $\Bbbk$ and $c\in {\rm Aut}(V\otimes V)$. The pair $(V,c)$ is called a braided vector space if $c$ satisfies
$$(c\otimes id)(id \otimes c)(c\otimes id)=(id \otimes c)(c\otimes id)(id \otimes c).$$
Given $\theta\in \Bbb N$ and $I_{\theta}=\{1,2,\dots \theta\}$. The braided vector space $(V,c)$ of dimension $\theta$ is of diagonal type if there exist a basis $\{x_{i}\mid i\in I_{\theta}\}$ of $V$ and $q=(q_{ij})_{i,j\in I_{\theta}}\in \Bbbk^{\theta \times \theta}$ such that $c(x_{i}\otimes x_{j})=q_{ij}x_{j}\otimes x_{i}$ for all $i,j \in I_{\theta}$. In such case, $q=(q_{ij})_{i,j\in I_{\theta}}$ is called a braiding matrix of $(V, c)$.
\end{definition}

\begin{definition} Let $H$ be a Hopf algebra with invertible antipode $\mathcal{S}$,
\begin{enumerate}
\item A left-right Yetter-Drinfel'd module $(V,\cdot,\rho)$ over $H$ is a left $H$-module and a right $H$-comodule $\rho:V\rightarrow V\otimes H$ satisfying the compatibility condition $$\rho_{}(h\cdot v)=h_{2}\cdot v_{0}\otimes h_{3} v_{1}\mathcal{S}^{-1}(h_{1}),$$ for $h \in H, v \in V$. The category of left-right Yetter-Drinfel'd modules over $H$ is denoted by ${ }_{H} \mathcal{Y} \mathcal{D}^{H}$.

\item A left-left Yetter-Drinfel'd module $(V,\cdot,\rho_{l})$ is a left $H$-module and a left $H$-comodule $\rho_{l}:V\rightarrow H\otimes V$ satisfying the compatibility condition $$\rho_{l}(h\cdot v)=h_{1}v_{-1}\mathcal{S}(h_{3})\otimes h_{2}\cdot v_{0},$$ for $h \in H, v \in V.$ The category of left-left Yetter-Drinfel'd modules over $H$ is denoted by ${ }_{H}^{H} \mathcal{Y} \mathcal{D}$.
\end{enumerate}
\end{definition}

Let $H$ be a Hopf algebra over $\Bbbk$ with invertible antipode $\mathcal{S}$, the Drinfel'd double $\mathcal{D}(H)=H^{\ast cop}\otimes H$ is a Hopf algebra with tensor product coalgebra structure and algebra structure defined by~$(\phi \otimes h)(\phi^{\prime} \otimes h^{\prime})=\phi \langle \phi_{3}^{\prime}, h_{1}\rangle  \phi_{2}^{\prime} \otimes h_{2}\langle \phi_{1}^{\prime},\mathcal{S}^{-1}(h_{3})\rangle h^{\prime}, ~~\forall \phi, \phi^{\prime}\in H^{\ast}, h, h^{\prime}\in H$. The following well-known result is due to \cite{Maji91}.
\begin{lemma}\label{lemma2.3}
Let $H$ be a finite-dimensional Hopf algebra. Then the category $_{H} \mathcal{Y} \mathcal{D}^{H}$ of left-right Yetter-Drinfel'd modules can be identified with the category ${}_{\mathcal{D}(H)}M$ of left modules over the Drinfel'd double $\mathcal{D}(H)$.
\end{lemma}
\begin{remark}\label{remark2.4}
\begin{enumerate}
\item Assume $(V, \cdot, \rho)$ is a left-right Yetter-Drinfel'd module over $H$, where $\rho_{}(v)=v_{0}\otimes v_{1}$ for $v \in V, h \in H$. Let $\rho_{l}(v)=S(v_{1})\otimes v_{0}$, then $\left(V, \cdot, \rho_{l}\right)$ is a left-left Yetter-Drinfel'd module over $H$.
\item The category ${ }_{H}^{H} \mathcal{Y} \mathcal{D}$ is a braided monoidal category with braidings $c_{V,W}(v\otimes w)=v_{-1}\cdot w\otimes v_{0}$ for $V,W\in { }_{H}^{H} \mathcal{Y} \mathcal{D}$, $v\in V,w\in W$.
\end{enumerate}
\end{remark}

\begin{definition}
 Let $V\in {}_{H}^{H} \mathcal{Y} \mathcal{D}$, a braided $\Bbb N$-graded Hopf algebra  $R=\oplus_{n\geq 0}R(n)$ in ${}_{H}^{H} \mathcal{Y} \mathcal{D}$ is called Nichols algebra of $V$ if
 \begin{enumerate}
 \item $\Bbbk\simeq R(0), V\simeq R(1)\in {}_{H}^{H} \mathcal{Y} \mathcal{D}.$
 \item $R(1)=P(R)=\{r\in R\mid \Delta_{R}(r)=r\otimes 1+1\otimes r\}.$
 \item $R$ is generated as algebra by $R(1)$.
 \end{enumerate}
 Then, $R$ is denoted by $\mathcal{B}(V)=\oplus_{n\geq 0}\mathcal{B}^{n}(V).$
\end{definition}

\begin{remark} \cite{AS,Sch1996,Shi} The Nichols algebra $\mathcal{B}(V)$ is completely determined by the braided vector space. There are the following maps:
 \begin{center}
 $\Omega_{n, 1}:=\mathrm{id}+c_{n}+c_{n} c_{n-1}+\cdots+c_{n} c_{n-1} \cdots c_{1}=\mathrm{id}+c_{n} \Omega_{n-1,1},$\\
 $\Omega_{1}:=\mathrm{id}, \quad \Omega_{2}:=\mathrm{id}+c, \quad\Omega_{n}:=\left(\Omega_{n-1} \otimes \mathrm{id}\right)\Omega_{n-1,1},$\\
\end{center}
where $\Omega_{n, 1} \in{\rm End}_{\Bbbk}\left(V^{\otimes(n+1)}\right),~\Omega_{n} \in {\rm End}_{\Bbbk}\left(V^{\otimes n}\right)$,
As a vector space, the Nichols algebra $\mathcal{B}(V)$ is equal to
$$\mathcal{B}(V)=\Bbbk \oplus V \oplus \bigoplus_{n=2}^{\infty} V^{\otimes n} / \operatorname{ker} \Omega_{n}.$$
\end{remark}
\subsection{Extensions}
We describe Hopf algebra $H_{c: \sigma_{0}}$ as an abelian extension and we refer to~\cite{AA,Kashina,Masuoka,TM} for related facts.
\begin{definition} Let $(H):K \stackrel{\iota}{\rightarrow} H \stackrel{\pi}{\rightarrow} A$ be a sequence of Hopf algebras and let $K^{+}$ be the kernel of the counit of $K$.
\begin{enumerate}
    \item Suppose $\iota$ is injective, $\pi$ is surjective. So that we may regard $\iota$ as an inclusion, $\pi$ as quotient. The sequence $(H)$ is called an extension of $H$ by $K$ if it satisfies $HK^{+}={\rm Ker}\pi.$
    \item An extension $(H)$ is called an abelian extension if $K$ is commutative and $A$ is cocommutative.
\end{enumerate}
\end{definition}
\begin{remark}\emph{Let $G$ be a finite group and $\Bbbk G$ be the group algebra with the usual Hopf algebra structure. Its dual Hopf algebra $(\Bbbk G)^{*}$ can be  identified with $\Bbbk^{G}$, the algebra of functions from $G$ to $\Bbbk$.
If a Hopf algebra $H$ can fit into an extension $\Bbbk^{G} \rightarrow H \rightarrow \Bbbk F,$ where $G,F$ are finite groups, then $H$ can be described as a bicrossed product $\Bbbk^{G}\#_{\sigma, \tau} \Bbbk F$: there exist maps}
$$G \stackrel{\triangleleft}{\leftarrow} G \times F \stackrel{\triangleright}{\rightarrow} F, \quad G \times F \times F \stackrel{\sigma}{\rightarrow} \Bbbk^{\times}, \quad G \times G \times F \stackrel{\tau}{\rightarrow} \Bbbk^{\times},$$
\emph{such that $(F, G, \triangleleft, \triangleright)$ is a matched pair of groups, $(\sigma, \tau)$ is a pair of compatible normal cocycles. That is, $(F, G, \triangleleft, \triangleright, \sigma, \tau)$ satisfies the following conditions},
\begin{equation*}
\begin{array}{l}
(t \triangleright g g^{\prime})=(t \triangleright g)((t \triangleleft g) \triangleright g^{\prime}),~t t^{\prime} \triangleleft g=(t \triangleleft(t^{\prime} \triangleright g))(t^{\prime} \triangleleft g),\\
\sigma(g \triangleleft t ; t^{\prime}, t^{\prime \prime}) \sigma(g ; t, t^{\prime} t^{\prime \prime})=\sigma(g ; t, t^{\prime})\sigma(g ; t t^{\prime}, t^{\prime \prime}),~~~
 \sigma\left(1 ; t, t^{\prime}\right) =\sigma\left(g ; 1, t^{\prime}\right)=\sigma(g ; t, 1)=1,\\
 \tau\left(g g^{\prime}, g^{\prime \prime} ; t\right) \tau\left(g, g^{\prime} ; g^{\prime \prime} \triangleright t\right) =\tau\left(g^{\prime}, g^{\prime \prime} ; t\right) \tau\left(g, g^{\prime} g^{\prime \prime} ; t\right),
 ~\tau\left(1, g^{\prime} ; t\right)=\tau(g, 1 ; t)=\tau\left(g, g^{\prime} ; 1\right)=1, \\
 \sigma\left(g g^{\prime} ; t, t^{\prime}\right) \tau\left(g, g^{\prime} ; t t^{\prime}\right)= \sigma\left(g ; g^{\prime} \triangleright t,\left(g^{\prime} \triangleleft t\right) \triangleright t^{\prime}\right) \sigma\left(g^{\prime} ; t, t^{\prime}\right)\tau\left(g, g^{\prime} ; t\right) \tau\left(g \triangleleft\left(g^{\prime} \triangleright t\right), g^{\prime} \triangleleft t ; t^{\prime}\right).
\end{array}
\end{equation*}
\emph{for all $g, g^{\prime}, g^{\prime \prime} \in G,~t, t^{\prime}, t^{\prime \prime} \in F.$ The product, coproduct and antipode of $H\simeq \Bbbk^{G} \#_{\sigma, \tau} \Bbbk F$ in the basis
$\left\{e_{g} \# t: g \in G, t \in F\right\}$ has the form:}
\begin{equation}\label{2-2}
\begin{array}{l}
\left(e_{g} \# t\right) \cdot\left(e_{g^{\prime}} \# t^{\prime}\right)=\delta_{g \triangleleft t^{}, g^{\prime}} \sigma\left(g ; t, t^{\prime}\right) e_{g} \# t t^{\prime}, \\
\Delta\left(e_{g} \# t\right)=\sum\limits_{g^{\prime} g^{\prime \prime}=g} \tau\left(g^{\prime}, g^{\prime \prime} ; t\right) e_{g^{\prime}} \#\left(g^{\prime \prime} \triangleright t\right) \otimes e_{g^{\prime \prime}} \# t,\\
\mathcal{S}\left(e_{g} \# t\right)=\sigma\left(g^{-1} ; g \triangleright t,(g \triangleright t)^{-1}\right)^{-1} \tau\left(g^{-1}, g ; t\right)^{-1} e_{(g \triangleleft t)^{-1}} \#(g \triangleright t)^{-1}.
\end{array}
\end{equation}
\end{remark}
Let $G=\{x,y\mid x^{4}=1=y^{2},xy=yx\}$, $\{e_{g}\}_{g\in G}$ the dual basis of $G$ in $\Bbbk^{G}$ and $t$ the generator of $\Bbb Z_{2}$.
In this paper, we study the following type abelian extension of $H=H_{c: \sigma_{0}}$,
$$\Bbbk^{G}\stackrel{\iota}{\rightarrow} H \stackrel{\pi}{\rightarrow} \Bbbk \Bbb Z_{2},$$
 equivalently, the Hopf algebra $H$ is determined by $(\Bbb Z_{2}, G, \triangleleft, \triangleright, \sigma, \tau)$ where
 \begin{align}\label{2-2}
  & \triangleright~\text{is a  trivial action},\quad\quad x \triangleleft t=x y,\quad\quad y \triangleleft t=y,\\
  &\sigma(x^{i} y^{j}, t, t)=(-1)^{\frac{i(i-1)}{2}},\quad\tau(x^{i} y^{j}, x^{k} y^{l}, t)=(-1)^{j k}.
\end{align}
We denote by $e_{\gamma}t$ the $e_{\gamma}\# t$ for $\gamma\in G$. As an algebra, $H$ is generated by $\{ e_{g},t \}_{g \in G}$ and Hopf algebra structure of $H$ is given by
\begin{equation}\label{3.26}
\begin{array}{l}
 e_{g}e_{h}=\delta_{g,h}e_{g},\ te_{g}=e_{g\triangleleft t} t,\ t^2=\sum\limits_{g \in G}\sigma(g,t,t)e_{g}, \\
 \Delta (e_{g})=\sum\limits_{ h,k \in G,\ hk=g} e_{h}\otimes e_{k},\ \Delta(t)=[\sum\limits_{g,h \in G}\tau(g,h,t)e_{g}\otimes e_{h}](t\otimes t),\\
   \mathcal{S}(t)=\sum\limits_{g\in G}\sigma(g^{-1},t,t)^{-1}\tau(g,g^{-1},t)^{-1}e_{(g\triangleleft t)^{-1}}t,\ \mathcal{S}(e_g)=e_{g^{-1}},\\
   \epsilon(t)=1,\ \epsilon(e_{g})=\delta_{g,1}1, ~~~~\forall g,h\in G.\\
\end{array}
\end{equation}

\begin{remark}
\emph{The semisimple Hopf algebras in \cite{Kashina} with group-like elements $G(H)=\Bbb Z_{4}\times \Bbb Z_{2}$ are the following 7 non-trivial ones described by the data $(F,G,\triangleleft,\sigma,\tau)$, and among of them, $H_{b:1}$ were studied by Zheng-Gao-Hu in \cite{Zheng21}.}
\begin{table}[h!]
    \centering
    \begin{tabular}{cc}
         \hline
$H_{a:1}$   &   ~~~~~ $x \triangleleft  t=x, y\triangleleft t=x^{2}y ;\sigma\left(x^{i} y^{j}, t \cdot t\right)=1, \tau\left(x^{i} y^{j}, x^{k} y^{l},t\right)=(-1)^{j k}.$\\
\hline
$H_{a:y}$   &   ~~~~~ $x \triangleleft  t=x, y\triangleleft t=x^{2}y ;\sigma\left(x^{i} y^{j}, t \cdot t\right)=(-1)^{j}, \tau\left(x^{i} y^{j}, x^{k} y^{l},t\right)=(-1)^{j k}.$\\
\hline
$H_{b:y}$    &   ~~~~~ $x \triangleleft  t=x^{3}, y\triangleleft t=y ;\sigma\left(x^{i} y^{j}, t \cdot t\right)=(-1)^{j}, \tau\left(x^{i} y^{j}, x^{k} y^{l},t\right)=(-1)^{j k}.$\\
\hline
 $H_{b:x^{2}y}$  &   ~~~~~ $x \triangleleft  t=x^{3}, y\triangleleft t=y ;\sigma\left(x^{i} y^{j}, t \cdot t\right)=(-1)^{i+j}, \tau\left(x^{i} y^{j}, x^{k} y^{l},t\right)=(-1)^{j k}.$\\
\hline
$H_{b:1}$        &   ~~~~~ $x \triangleleft  t=x^{3}, y\triangleleft t=y ;\sigma\left(x^{i} y^{j}, t \cdot t\right)=1, \tau\left(x^{i} y^{j}, x^{k} y^{l},t\right)=(-1)^{j k}.$\\
\hline
$H_{c: \sigma_{0}}$ &  ~~~~~$x \triangleleft t=x y, y \triangleleft t=y$;
                      $\sigma(x^{i} y^{j}, t, t)=(-1)^{\frac{i(i-1)}{2}}$, $\tau(x^{i} y^{j}, x^{k} y^{l}, t)=(-1)^{j k}.$ \\
	            \hline
$H_{c: \sigma_{1}}$ &   ~~~~~ $x \triangleleft  t=x y, y\triangleleft t=y ;\sigma\left(x^{k} y^{l}, t \cdot t\right)=(-1)^{\frac{k(k-1)}{2}} {\rm i}^{k},  {\rm i}^{2}=-1,
\tau\left(x^{i} y^{j}, x^{k} y^{l},t\right)=(-1)^{j k}.$\\
         \hline
    \end{tabular}
\end{table}
\end{remark}
\section{The Drinfel'd double of $H_{c: \sigma_{0}}$}
In the following of this paper, we let $H$ be the Hopf algebra $H_{c: \sigma_{0}}$ and $G=\{x,y\mid x^{4}=1=y^{2},xy=yx\}$. In this section, we want to describe the Drinfel'd double $\mathcal{D}(H)$. For this purpose, we define the following elements in $H^{*}$:
 \begin{equation}\label{x}
 \begin{array}{c}
 \zeta_{g}: \left\{\begin{array}{ll}
 e_{h}\mapsto \delta_{g,h}\\
 e_{h} t\mapsto 0,  \\
\end{array}
\right.\end{array}\quad
\begin{array}{c}
 \chi_{h}: \left\{\begin{array}{ll}
 e_{g}\mapsto  0\\
 e_{g} t\mapsto  \delta_{g,h}  \\
\end{array}
\right.\end{array}~~~( g, h\in G).
\end{equation}

Deducing from the coalgebra structure of $H$, it is not hard to see that the algebraic structure of $H^{*}$ is given by 
\begin{equation}\label{a}
\begin{array}{l}
\zeta_{x^{i}y^{j}} \zeta_{x^{k}y^{l}}= \zeta_{x^{i+k}y^{j+l}},\quad\chi_{x^{i}y^{j}}\chi_{x^{k}y^{l}}=(-1)^{jk}\chi_{x^{i+j}y^{j+l}},\\
\zeta_{x^{i}y^{j}}\chi_{x^{k}y^{l}}=0=\chi_{x^{k}y^{l}}\zeta_{x^{i}y^{j}},\quad 1_{H^{*}}=\varepsilon=\zeta_{1}+\chi_{1}.\\
\end{array}
\end{equation}

 Since $e_{x^{i}y^{j}},~t~(0\leq i\leq 3,~0\leq j\leq 1)$ are the generators of $H$ and $\zeta_{x^{i}y^{j}},~\chi_{x^{i}y^{j}}~(0\leq i\leq 3,~0\leq j\leq 1)$ are the generators of $H^{\ast}$, the Drinfel'd double $\mathcal{D}(H)=H^{\ast cop}\bowtie H$ is generated by
$\zeta_{x^{i}y^{j}} \bowtie 1_{H}^{}$,~$\chi_{x^{i}y^{j}} \bowtie 1_{H}^{}$,~$1_{H^{*cop}}\bowtie e_{x^{i}y^{j}}$,~$1_{H^{*cop}}\bowtie t$, and we abbreviate them by $\zeta_{x^{i}y^{j}},~\chi_{x^{i}y^{i}},~e_{x^{i}y^{j}},~t$.
\begin{proposition}\label{theorem3.1}
As an algebra, $\mathcal{D}(H)$ is generated by $e_{x^{i}y^{j}},~t,~\zeta_{x^{i}y^{j}},~\chi_{x^{i}y^{j}},~(0\leq i\leq 3,~0\leq j\leq 1)$ with relations:
\begin{align}\label{3-3}
&\zeta_{x^{i}y^{j}} \zeta_{x^{k}y^{l}}= \zeta_{x^{i+k}y^{j+l}},\quad \chi_{x^{i}y^{j}}\chi_{x^{k}y^{l}}=(-1)^{jk}\chi_{x^{i+k}y^{j+l}},\quad
\zeta_{x^{i}y^{j}}\chi_{x^{k}y^{l}}=0=\chi_{x^{k}y^{l}}\zeta_{x^{i}y^{j}},
\end{align}
\begin{align}\label{3-4}
&1=\sum\limits_{g\in G} e_{g}, \quad e_{g} e_{g^{\prime}}=\delta_{g, g^{\prime}}e_{g},\quad te_{g}=e_{g\triangleleft t} t,\quad t^{2}=\sum\limits_{g}\sigma(g;t,t)e_{g},\quad(g, g^{\prime}\in G),
\end{align}
\begin{align}\label{3-5}
e_{g}\chi_{h}=\chi_{h}e_{gh(h^{-1}\triangleleft t)},\quad \quad e_{g}\zeta_{h}=\zeta_{h}e_{g},\quad  ( g, h\in G),
\end{align}
\begin{equation}\label{3-6}
\begin{array}{l}
t\zeta_{h}=\zeta_{h\triangleleft t}(\sum\limits_{g\in G}\tau\left(h\triangleleft t, g ; t\right)\tau(~(h\triangleleft t)g, h^{-1}\triangleleft t; t)\tau(h\triangleleft t,h^{-1}\triangleleft t;~t)e_{g} t),\quad (  h\in G),\\
t\chi_{h}=\chi_{h\triangleleft t}(\sum\limits_{g\in G}\tau\left(h, g ; t\right)\tau(hg, h^{-1}\triangleleft t; t)\tau(h\triangleleft t,h^{-1}\triangleleft t;~t)e_{g} t),\quad( h\in G).
\end{array}
\end{equation}
\end{proposition}
\begin{proof}
We only verify two equations of (\ref{3-5}) and (\ref{3-6}). Due to the definition, it follows that,
$$hf=(1_{H^{*cop}}\bowtie h)(f\bowtie 1_{H})=h_{1}\rightharpoonup f \leftharpoonup \mathcal{S}^{-1}(h_{3})\bowtie h_{2}=\langle f_{1}, \mathcal{S}^{-1}(h_{3})\rangle \langle {f_{3},h_{1}} \rangle f_{2}\bowtie h_{2},$$
where $h\in H,f\in H^{*cop}$. Let $k\in H$, it is straightforward to show that
$$\langle\langle f_{1}, \mathcal{S}^{-1}(h_{3})\rangle \langle f_{3},h_{1}\rangle f_{2}, k\rangle=\langle f,  \mathcal{S}^{-1}(h_{3})\cdot k\cdot h_{1}\rangle.$$
Our problem reduces to determine $\mathcal{S}^{-1}(h_{3})\cdot k\cdot h_{1}$ in each case.
Assume $\gamma,g,h\in G$ and $k=e_{\gamma}$ or ${e_{\gamma}}t$, after some tedious manipulation by (\ref{2-2}), it follows that,
\begin{align}\label{3-8}
&\mathcal{S}^{-1}(~e_{{g}_{3}}~)\cdot~e_{\gamma}\cdot~e_{{g}_{1}}=\delta_{{g_{3}}^{-1},\gamma}~e_{\gamma}\cdot~e_{{g}_{1}}=\delta_{{g_{3}}^{-1},\gamma}\delta_{\gamma,g_{1}}~e_{\gamma},
\end{align}
\begin{align*}
&\mathcal{S}^{-1}(~e_{g_{3}}~ t)=\sigma(~{g_{3}}^{-1};~t,~t)\tau(~{g_{3}}^{-1},~{g_{3}}; t)e_{(g_{3}\triangleleft t)^{-1}}~ t,
\end{align*}
\begin{align}\label{3-9}
&\mathcal{S}^{-1}(~e_{g_{3}}  t~)\cdot~(e_{\gamma}~ t)~\cdot~(e_{g_{1}} ~ t)=\tau(~{g_{3}}^{-1},~g_{3}; t)\delta_{({(g_{3}\triangleleft t)^{-1}})\triangleleft t,\gamma}\delta_{(g_{3}\triangleleft t)^{-1},g_{1}}(e_{\gamma\triangleleft t}t),
\end{align}
\begin{align*}
&\Delta(t)=\Delta(1\cdot~t)=\Delta\left(\Sigma_{g}e_{g} ~ t\right)=\sum\limits_{g_1,g_2\in G} \tau\left(g_{1}, g_{2} ; t\right) (e_{g_{1}} ~ t) \otimes (e_{g_{2}} ~ t),
\end{align*}
\begin{align}\label{3-10}
\Delta_{2}(t)=(\Delta \otimes id)(\Delta(t))=\sum\limits_{g_{1},g_{2},g_{3}\in G}\tau\left(g_{1}, g_{2} ; t\right)
\tau(g_{1}g_{2}, g_{3}; t)(e_{g_{1}}~ t) \otimes (e_{g_{2}} ~ t)\otimes (e_{g_{3}} ~ t).
\end{align}
Let $k=e_{\gamma}$, by~(\ref{3-8}), it follows that
\begin{align*}
&\langle\zeta_h,\mathcal{S}^{-1}(~e_{{g}_{3}}~)\cdot~e_{\gamma}\cdot~e_{{g}_{1}}\rangle=\delta_{{g_{3}}^{-1},\gamma}\delta_{\gamma,g_{1}}\delta_{\gamma,h},\\
&\langle(\zeta_{h})_{1}, \mathcal{S}^{-1}(e_{g_{3}})\rangle\langle(\zeta_{h})_{3},(e_{g_{1}})\rangle(\zeta_{h})_{2}=\zeta_{h},\\
&e_{g}\zeta_{h}=\zeta_{h}e_{g},~~(g_{3}^{-1}=g_{1}=\gamma=h,~ e_{g_{2}}=e_{g}).
\end{align*}
It is clear that relation holds for $k=e_{\gamma}~ t$.

Let $k=e_{\gamma}~ t$, due to~(\ref{3-9}) and (\ref{3-10}), it follows that
\begin{align*}
&\langle\chi_h,\tau(g_{1}, g_{2} ; t)\tau(g_{1}g_{2}, g_{3}; t)\mathcal{S}^{-1}(~e_{g_{3}}~ t~)\cdot~(e_{\gamma}~ t)~\cdot~(e_{g_{1}} ~ t)\rangle\\
&=\tau(g_{1},g_{2};t)\tau(g_{1}g_{2},g_{3};t)\tau({g_{3}}^{-1},{g_{3}};t)\delta_{({(g_{3}\triangleleft t)^{-1}})\triangleleft t,\gamma}\delta_{(g_{3}\triangleleft t)^{-1},g_{1}}\delta_{h\triangleleft t, \gamma},\\
&\langle (\chi_{h})_{1}, \mathcal{S}^{-1}({e_{g}}_{3}~t)\rangle\langle(\chi_{h})_{3},{e_{g}}_{1}~t\rangle(\chi_{h})_{2}=\chi_{h\triangleleft t},\\
&t\chi_{h}=\chi_{h\triangleleft t}(\sum\limits_{g\in G}\tau\left(h, g ; t\right)\tau(hg, h^{-1}\triangleleft t; t)\tau(h\triangleleft t,h^{-1}\triangleleft t;~t)(e_{g} ~ t)),\\
&~~~({g_{3}}^{-1}\triangleleft t=\gamma\triangleleft t=h=g_{1},~g_{2}=g).
\end{align*}
Also, it is not hard to see that relation holds for $k=e_{\gamma}$.
\end{proof}

\section{The Category of ${ }_{H}^{H} \mathcal{Y} \mathcal{D}$}
In the previous section, we already described the algebraic structure of the Drinfel'd double $\mathcal{D}(H)$. In this section, we want to give all simple left $\mathcal{D}(H)$-modules, equivalently, give all  simple Yetter-Drinfel'd modules $V$ over $H$.
\subsection{The simple left $\mathcal{D}(H)$\mbox{-}module}

 Let $V$ be a left $\mathcal{D}(H)$\mbox{-}module, we define $V_{\mu}=\{v\in V\mid e_{h}v=\delta_{\mu,h}v,\;h\in G\}$ for $\mu \in G$, then $V$ decomposes into a direct sum of $V_{\mu}$'s as vector space. Furthermore, by the relations $\zeta_{g}\chi_{h}=0=\chi_{h}\zeta_{g}~(g, h \in G)$ and $1_{H^{*}}=\varepsilon=\zeta_{1}+\chi_{1}$, it follows that as a vector space $V_{\mu}$ decomposes into a direct sum of $\zeta_{1}V_{\mu}$ and $\chi_{1}V_{\mu}.$ Let $S=\{1,y,x^{2},x^{2}y\}$ and~$T=\{x,xy,x^{3},x^{3}y\}$ be the subsets of $G$. By~(\ref{2-2}), it follows that $T$ and $S$ are stable under the action $\triangleleft$ of $t$. Furthermore, the elements of $S$ are fixed points under the action above.
 We denote by $\xi$ a primitive $4$-th root of $1$. We construct the following 32 one-dimensional spaces together with the corresponding $\mathcal{D}(H)$-action:
\begin{enumerate}
\item[(1)] $V_{1,i}^{+}~(i=0,1,2,3):$ the action of $\mathcal{D}(H)$ is given by
\begin{align*}
e_{x^{k}y^{l}}\mapsto \delta_{x^{k}y^{l},1}, \quad t\mapsto 1, \quad \zeta_{x^{k}y^{l}}\mapsto \xi^{ik},\quad \chi_{x^{k}y^{l}}\mapsto 0,~(0\leq k\leq 3,~0\leq l\leq 1).
\end{align*}

\item[(2)] $V_{1,i}^{-}~(i=0,1,2,3):$ the action of $\mathcal{D}(H)$ is given by
\begin{align*}
e_{x^{k}y^{l}}\mapsto \delta_{x^{k}y^{l},1}, \quad t\mapsto -1, \quad \zeta_{x^{k}y^{l}}\mapsto \xi^{ik},\quad \chi_{x^{k}y^{l}}\mapsto 0,~(0\leq k\leq 3,~0\leq l\leq 1).
\end{align*}

\item[(3)] $V_{y,i}^{+}~(i=0,1,2,3):$ the action of $\mathcal{D}(H)$ is given by
\begin{align*}
e_{x^{k}y^{l}}\mapsto \delta_{x^{k}y^{l},y}, \quad t\mapsto 1, \quad \zeta_{x^{k}y^{l}}\mapsto \xi^{ik}(-1)^{l},\quad \chi_{x^{k}y^{l}}\mapsto 0,~(0\leq k\leq 3,~0\leq l\leq 1).
\end{align*}

\item[(4)] $V_{y,i}^{-}~(i=0,1,2,3):$ the action of $\mathcal{D}(H)$ is given by
\begin{align*}
e_{x^{k}y^{l}}\mapsto \delta_{x^{k}y^{l},y}, \quad t\mapsto -1, \quad \zeta_{x^{k}y^{l}}\mapsto \xi^{ik}(-1)^{l},\quad \chi_{x^{k}y^{l}}\mapsto 0,~(0\leq k\leq 3,~0\leq l\leq 1).
\end{align*}

\item[(5)] $V_{x^{2},i}^{+}~(i=0,1,2,3):$ the action of $\mathcal{D}(H)$ is given by
\begin{align*}
e_{x^{k}y^{l}}\mapsto \delta_{x^{k}y^{l},x^{2}}, \quad t\mapsto \xi, \quad \zeta_{x^{k}y^{l}}\mapsto \xi^{ik},\quad \chi_{x^{k}y^{l}}\mapsto 0,~(0\leq k\leq 3,~0\leq l\leq 1).
\end{align*}

\item[(6)] $V_{x^{2},i}^{-}~(i=0,1,2,3):$ the action of $\mathcal{D}(H)$ is given by
\begin{align*}
e_{x^{k}y^{l}}\mapsto \delta_{x^{k}y^{l},x^{2}}, \quad t\mapsto -\xi, \quad \zeta_{x^{k}y^{l}}\mapsto \xi^{ik},\quad \chi_{x^{k}y^{l}}\mapsto 0,~(0\leq k\leq 3,~0\leq l\leq 1).
\end{align*}

\item[(7)] $V_{x^{2}y,i}^{+}~(i=0,1,2,3):$ the action of $\mathcal{D}(H)$ is given by
\begin{align*}
e_{x^{k}y^{l}}\mapsto \delta_{x^{k}y^{l},x^{2}y}, \quad t\mapsto \xi, \quad \zeta_{x^{k}y^{l}}\mapsto \xi^{ik}(-1)^{l},\quad \chi_{x^{k}y^{l}}\mapsto 0,~(0\leq k\leq 3,~0\leq l\leq 1).
\end{align*}

\item[(8)] $V_{x^{2}y,i}^{-}~(i=0,1,2,3):$ the action of $\mathcal{D}(H)$ is given by
\begin{align*}
e_{x^{k}y^{l}}\mapsto \delta_{x^{k}y^{l},x^{2}y}, \quad t\mapsto -\xi, \quad \zeta_{x^{k}y^{l}}\mapsto \xi^{ik}(-1)^{l},\quad \chi_{x^{k}y^{l}}\mapsto 0,~(0\leq k\leq 3,~0\leq l\leq 1).
\end{align*}
\end{enumerate}
\begin{enumerate}
\item[(S0)]~We denote them through a unified way $V_{\mu,i}^{\pm}$ for $\mu\in \{1,~y,~x^{2},~x^{2}y\},\;0\leq i\leq 3.$
\end{enumerate}

\begin{lemma}\label{lemma4.1}
 Above $32$ one-dimensional spaces together with the corresponding action are $\mathcal{D}(H)$\mbox{-}modules. Furthermore, all one-dimensional $\mathcal{D}(H)$\mbox{-}modules belong to$(S0)$.

\begin{proof}
It is straightforward to show that $V_{\mu,i}^{\pm}$ is a simple module through checking the relations given in Proposition \ref{theorem3.1}. Assume $V$ is a one-dimensional $\mathcal{D}(H)$\mbox{-}module. Take a basis $\{v\}$ of $V$. Since $1=\sum\limits_{k,l} e_{x^{k}y^{l}}$ is a decomposition of $1$ into orthogonal idempotents, it follows that there exists $\mu\in G$ such that $e_{x^{k}y^{l}}\cdot v=\delta_{x^{k}y^{l},\mu}v$, equivalently, $v=e_{\mu}v$ and $V=V_{\mu}$. Since $te_{g}v=e_{g\triangleleft t}tv~(g\in G)$, we have $t\cdot v\in V_{\mu\triangleleft t}$ and this leads to $\mu\triangleleft t=\mu$, equivalently, $\mu\in S$. Assume $V=\chi_{1}V_{u}$ and $\rm dim V=1$. But in this situation, $e_{g}\chi_{h}=\chi_{h}e_{gh(h^{-1}\triangleleft t)}$, it induces $\chi_{h}v_{\mu}\in V_{\mu (h \triangleleft t)h^{-1}}$. In particular, $\chi_{x}v_{\mu}\in V_{\mu y}$, this contradicts the fact $\rm dim V=1$. Thus, we have
    \begin{align}
v=\zeta_{1}v,\quad\quad V=\zeta_{1}V_{\mu}~(\mu\in S),\quad \quad \chi_{x^{k}y^{l}}\cdot v=0,\quad t\cdot v=\lambda v~(\lambda\in \Bbbk).
  \end{align}
According to (\ref{3-4}), it follows that $\lambda^{2}v=t^{2}v=\sum\limits_{g}\sigma(g,t,t)e_{g}v$. Thus,
 $$\lambda^{2}=\sigma(\mu,t,t),\quad  \lambda=\pm 1~(\mu=1,~y),\quad\lambda=\pm \xi~(\mu=x^{2},~x^{2}y).$$
Assume $\zeta_{x^{}}v=m v,~\zeta_{y}v=n v,$ where $m,n\in \Bbbk$.  By relation~(\ref{3-3}), it follows that
\begin{align}
(\zeta_{x})^{4}=\zeta_{x^{4}}=\zeta_{1},\quad\quad (\zeta_{y})^{2}=\zeta_{y^{2}}=\zeta_{1}.
\end{align}
Thus, we have
 \begin{equation} \label{equation4.80}
 v\stackrel{(4.1)}{=}\zeta_{1}v=\zeta_{x^{4}}v\stackrel{(4.2)}{=}(\zeta_{x})^{4}v=m^{4}v,\quad v\stackrel{(4.1)}{=}\zeta_{1}v=\zeta_{y^{2}}v\stackrel{(4.2)}{=}(\zeta_{y})^{2}v=n^{2}v.
 \end{equation}

 According to the relation~(\ref{3-6}), it follows that
\begin{equation*}
\begin{array}{l}
t\zeta_{x}v=\zeta_{x\triangleleft t}(\sum\limits_{g}\tau\left(x\triangleleft t, g ; t\right)\tau(~(x\triangleleft t)g, x^{-1}\triangleleft t; t)\tau(x\triangleleft t,x^{-1}\triangleleft t;~t)(e_{g} t))v,\\
t\zeta_{y}v=\zeta_{y\triangleleft t}(\sum\limits_{g}\tau\left(y\triangleleft t, g ; t\right)\tau(~(y\triangleleft t)g,y^{-1}\triangleleft t; t)\tau(y\triangleleft t,y^{-1}\triangleleft t;~t)(e_{g}  t))v.\\
\end{array}
\end{equation*}
Equivalently,
\begin{equation}\label{4-2-0}
\begin{array}{l}
n\tau\left(xy, \mu ; t\right)\tau((xy)\mu, x^{3}y; t)\tau(xy,x^{3}y;~t)=1
\end{array}
\end{equation}
Due to (\ref{equation4.80})~(\ref{4-2-0}), it follows that
\begin{equation}\label{4-2}
\begin{array}{l}
m=\xi^{i},~i\in\{0,1,2,3\},\\
n=1,~(\mu=1,~x^{2}),\\
n=-1,~(\mu=y,~x^{2}y).\\
\end{array}
\end{equation}
By the above arguments, it follows that there are exactly 32 non-isomorphic one-dimensional simple left $\mathcal{D}(H)$\mbox{-}modules, as shown above.
\end{proof}
\end{lemma}

Next, we turn to two-dimensional case. Let $i=0,1,2,3$, we construct the following 16 two-dimensional spaces $\k \langle v,~w\rangle$ together with the corresponding $\mathcal{D}(H)$-action:
\begin{enumerate}
\item[(1)] $W_{x,i,+}^{+}:~$ Let $0\leq k\leq 3,~0\leq l\leq 1$,~we define~$\chi_{x^{k}y^{l}}\mapsto 0,$
\begin{align*}
e_{x^{k}y^{l}}\mapsto \left(\begin{array}{ll}\delta_{x^{k}y^{l},x} & 0 \\
0 & \delta_{x^{k}y^{l},xy}\end{array}\right),~t\mapsto \left(\begin{array}{ll}0 & 1 \\
1 & 0\end{array}\right),~\zeta_{x}\mapsto \left(\begin{array}{ll}\xi^{i} & 0 \\
0 & \xi^{i}\end{array}\right),~\zeta_{y}\mapsto \left(\begin{array}{ll}-1 & 0 \\
0 & 1\end{array}\right).
\end{align*}
\item[(2)] $W_{x,i,-}^{+}:$ Let $0\leq k\leq 3,~0\leq l\leq 1$,~we define~$\chi_{x^{k}y^{l}}\mapsto 0,$
\begin{align*}
e_{x^{k}y^{l}}\mapsto \left(\begin{array}{ll}\delta_{x^{k}y^{l},x} & 0 \\
0 & \delta_{x^{k}y^{l},xy}\end{array}\right),~t\mapsto \left(\begin{array}{ll}0 & 1 \\
1 & 0\end{array}\right),~\zeta_{x}\mapsto \left(\begin{array}{ll}-\xi^{i} & 0 \\
0 & \xi^{i}\end{array}\right),~\zeta_{y}\mapsto \left(\begin{array}{ll}1 & 0 \\
0 & -1\end{array}\right).
\end{align*}

\item[(3)] $W_{x^{3},i,+}^{-}:$ Let $0\leq k\leq 3,~0\leq l\leq 1$,~we define~$\chi_{x^{k}y^{l}}\mapsto 0,$
\begin{align*}
e_{x^{k}y^{l}}\mapsto \left(\begin{array}{ll}\delta_{x^{k}y^{l},x^{3}} & 0 \\
0 & \delta_{x^{k}y^{l},x^{3}y}\end{array}\right),~t\mapsto \left(\begin{array}{ll}0 & -1 \\
1 & 0\end{array}\right),~\zeta_{x}\mapsto \left(\begin{array}{ll}\xi^{i} & 0 \\
0 & \xi^{i}\end{array}\right),~\zeta_{y}\mapsto \left(\begin{array}{ll}-1 & 0 \\
0 & 1\end{array}\right).
\end{align*}

\item[(4)] $W_{x^{3},i,-}^{-}:$ Let $0\leq k\leq 3,~0\leq l\leq 1$,~we define~$\chi_{x^{k}y^{l}}\mapsto 0,$
\begin{align*}
e_{x^{k}y^{l}}\mapsto \left(\begin{array}{ll}\delta_{x^{k}y^{l},x^{3}} & 0 \\
0 & \delta_{x^{k}y^{l},x^{3}y}\end{array}\right),~t\mapsto \left(\begin{array}{ll}0 & -1 \\
1 & 0\end{array}\right),~\zeta_{x}\mapsto \left(\begin{array}{ll}-\xi^{i} & 0 \\
0 & \xi^{i}\end{array}\right),~\zeta_{y}\mapsto \left(\begin{array}{ll}1 & 0 \\
0 & -1\end{array}\right).
\end{align*}
\end{enumerate}
\begin{enumerate}
\item[(S1)]~We denote them through a unified way $W_{\mu,i,\pm}^{\pm}$ for $\mu\in \{x,~x^{3}\},\;0\leq i\leq 3.$
\end{enumerate}

Let $i=0,1$, we construct the following 8 two-dimensional spaces $\k \langle v,~w\rangle$ together with the corresponding $\mathcal{D}(H)$-action:
\begin{enumerate}
\item[(1)] $X_{1,i}^{+}:$~Let $0\leq k\leq 3,~0\leq l\leq 1$,~we define~$\chi_{x^{k}y^{l}}\mapsto 0,$
\begin{align*}
e_{x^{k}y^{l}}\mapsto \left(\begin{array}{ll}\delta_{x^{k}y^{l},1} & 0 \\
0 & \delta_{x^{k}y^{l},1}\end{array}\right),~t\mapsto \left(\begin{array}{ll}0 & 1 \\
1 & 0\end{array}\right),~\zeta_{x}\mapsto \left(\begin{array}{ll}\xi^{i} & 0 \\
0 & -\xi^{i}\end{array}\right),~\zeta_{y}\mapsto \left(\begin{array}{ll}-1 & 0 \\
0 & -1\end{array}\right).
\end{align*}
\item[(2)] $X_{y,i}^{+}:$~Let $0\leq k\leq 3,~0\leq l\leq 1$,~we define~$\chi_{x^{k}y^{l}}\mapsto 0,$
\begin{align*}
e_{x^{k}y^{l}}\mapsto \left(\begin{array}{ll}\delta_{x^{k}y^{l},y} & 0 \\
0 & \delta_{x^{k}y^{l},y}\end{array}\right),~t\mapsto \left(\begin{array}{ll}0 & 1 \\
1 & 0\end{array}\right),~\zeta_{x}\mapsto \left(\begin{array}{ll}\xi^{i} & 0 \\
0 & -\xi^{i}\end{array}\right),~\zeta_{y}\mapsto \left(\begin{array}{ll}1 & 0 \\
0 & 1\end{array}\right).
\end{align*}
\item[(3)] $X_{x^{2},i}^{-}:$ Let $0\leq k\leq 3,~0\leq l\leq 1$,~we define~$\chi_{x^{k}y^{l}}\mapsto 0,$
\begin{align*}
e_{x^{k}y^{l}}\mapsto \left(\begin{array}{ll}\delta_{x^{k}y^{l},x^{2}} & 0 \\
0 & \delta_{x^{k}y^{l},x^{2}}\end{array}\right),~t\mapsto \left(\begin{array}{ll}0 & -1 \\
1 & 0\end{array}\right),~\zeta_{x}\mapsto \left(\begin{array}{ll}\xi^{i} & 0 \\
0 & -\xi^{i}\end{array}\right),~\zeta_{y}\mapsto \left(\begin{array}{ll}-1 & 0 \\
0 & -1\end{array}\right).
\end{align*}
\item [(4)]$X_{x^{2}y,i}^{-}:$ Let $0\leq k\leq 3,~0\leq l\leq 1$,~we define~$\chi_{x^{k}y^{l}}\mapsto 0,$
\begin{align*}
e_{x^{k}y^{l}}\mapsto \left(\begin{array}{ll}\delta_{x^{k}y^{l},x^{2}y} & 0 \\
0 & \delta_{x^{k}y^{l},x^{2}y}\end{array}\right),~t\mapsto \left(\begin{array}{ll}0 & -1 \\
1 & 0\end{array}\right),~\zeta_{x}\mapsto \left(\begin{array}{ll}\xi^{i} & 0 \\
0 & -\xi^{i}\end{array}\right),~\zeta_{y}\mapsto \left(\begin{array}{ll}1 & 0 \\
0 & 1\end{array}\right).
\end{align*}
\end{enumerate}
\begin{enumerate}
\item[(S2)]~We denote them through a unified way $X_{\mu,i}^{\pm}$ for $\mu\in \{1,~y,~x^{2},~x^{2}y\},\;i=0,1.$
\end{enumerate}

Then we construct the following 16 two-dimensional spaces $\k \langle v,~w\rangle$ together with the corresponding $\mathcal{D}(H)$-action:

\begin{enumerate}
\item[(1)]  $Y_{x,1}:~$ Let $0\leq k\leq 3$ and $0\leq l\leq 1$, we define $\zeta_{x^{k}y^{l}}\mapsto 0$,
\begin{align*}
e_{x^{k}y^{l}}\mapsto \left(\begin{array}{ll}\delta_{x^{k}y^{l},x} & 0 \\
0 & \delta_{x^{k}y^{l},xy}\end{array}\right), \quad t\mapsto \left(\begin{array}{ll}0 & -\xi \\
\xi & 0\end{array}\right),~\chi_{x}\mapsto \left(\begin{array}{ll}0 & 1 \\
 1&0\end{array}\right),~\chi_{y}\mapsto \left(\begin{array}{ll}1 & 0 \\
 0&-1\end{array}\right).
\end{align*}

\item[(2)] $Y_{x,2}:$ Let $0\leq k\leq 3$ and $0\leq l\leq 1$, we define $\zeta_{x^{k}y^{l}}\mapsto 0$,
\begin{align*}
e_{x^{k}y^{l}}\mapsto \left(\begin{array}{ll}\delta_{x^{k}y^{l},x} & 0 \\
0 & \delta_{x^{k}y^{l},xy}\end{array}\right), \quad t\mapsto \left(\begin{array}{ll}0 & \xi \\
-\xi & 0\end{array}\right),\quad \chi_{x}\mapsto \left(\begin{array}{ll}0 & 1 \\
 1&0\end{array}\right),~\chi_{y}\mapsto \left(\begin{array}{ll}1 & 0 \\
 0&-1\end{array}\right).
\end{align*}

\item[(3)]  $Y_{x,3}:~$ Let $0\leq k\leq 3$ and $0\leq l\leq 1$, we define $\zeta_{x^{k}y^{l}}\mapsto 0$,
\begin{align*}
e_{x^{k}y^{l}}\mapsto \left(\begin{array}{ll}\delta_{x^{k}y^{l},x} & 0 \\
0 & \delta_{x^{k}y^{l},xy}\end{array}\right), \quad t\mapsto \left(\begin{array}{ll}0 & -\xi \\
\xi & 0\end{array}\right),~\chi_{x}\mapsto \left(\begin{array}{ll}0 & -1 \\
 1&0\end{array}\right),~\chi_{y}\mapsto \left(\begin{array}{ll}-1 & 0 \\
 0&1\end{array}\right).
\end{align*}

\item[(4)] $Y_{x,4}:$ Let $0\leq k\leq 3$ and $0\leq l\leq 1$, we define $\zeta_{x^{k}y^{l}}\mapsto 0$,
\begin{align*}
e_{x^{k}y^{l}}\mapsto \left(\begin{array}{ll}\delta_{x^{k}y^{l},x} & 0 \\
0 & \delta_{x^{k}y^{l},xy}\end{array}\right),~t\mapsto \left(\begin{array}{ll}0 & \xi \\
-\xi & 0\end{array}\right),~\chi_{x}\mapsto \left(\begin{array}{ll}0 & -1 \\
 1&0\end{array}\right),~\chi_{y}\mapsto \left(\begin{array}{ll}-1 & 0 \\
 0&1\end{array}\right).
\end{align*}

\item [(5)] $Y_{x,5}:~$ Let $0\leq k\leq 3$ and $0\leq l\leq 1$, we define $\zeta_{x^{k}y^{l}}\mapsto 0$,
\begin{align*}
e_{x^{k}y^{l}}\mapsto \left(\begin{array}{ll}\delta_{x^{k}y^{l},x} & 0 \\
0 & \delta_{x^{k}y^{l},xy}\end{array}\right),~t\mapsto \left(\begin{array}{ll}0 & 1 \\
1 & 0\end{array}\right),~\chi_{x}\mapsto \left(\begin{array}{ll}0 & -1 \\
 1&0\end{array}\right),~\chi_{y}\mapsto \left(\begin{array}{ll}1 & 0 \\
 0&-1\end{array}\right).
\end{align*}

\item[(6)] $Y_{x,6}:$ Let $0\leq k\leq 3$ and $0\leq l\leq 1$, we define $\zeta_{x^{k}y^{l}}\mapsto 0$,
\begin{align*}
e_{x^{k}y^{l}}\mapsto \left(\begin{array}{ll}\delta_{x^{k}y^{l},x} & 0 \\
0 & \delta_{x^{k}y^{l},xy}\end{array}\right),~t\mapsto \left(\begin{array}{ll}0 & -1 \\
-1 & 0\end{array}\right),~\chi_{x}\mapsto \left(\begin{array}{ll}0 & -1 \\
 1&0\end{array}\right),~\chi_{y}\mapsto \left(\begin{array}{ll}1 & 0 \\
 0&-1\end{array}\right).
\end{align*}

\item [(7)]$Y_{x,7}:$ Let $0\leq k\leq 3$ and $0\leq l\leq 1$, we define $\zeta_{x^{k}y^{l}}\mapsto 0$,
\begin{align*}
e_{x^{k}y^{l}}\mapsto \left(\begin{array}{ll}\delta_{x^{k}y^{l},x} & 0 \\
0 & \delta_{x^{k}y^{l},xy}\end{array}\right),~t\mapsto \left(\begin{array}{ll}0 & 1 \\
1 & 0\end{array}\right)~\chi_{x}\mapsto \left(\begin{array}{ll}0 & 1 \\
 1&0\end{array}\right),~\chi_{y}\mapsto \left(\begin{array}{ll}-1 & 0 \\
 0&1\end{array}\right).
\end{align*}

\item[(8)] $Y_{x,8}:$ Let $0\leq k\leq 3$ and $0\leq l\leq 1$, we define $\zeta_{x^{k}y^{l}}\mapsto 0$,
\begin{align*}
e_{x^{k}y^{l}}\mapsto \left(\begin{array}{ll}\delta_{x^{k}y^{l},x} & 0 \\
0 & \delta_{x^{k}y^{l},xy}\end{array}\right),~t\mapsto \left(\begin{array}{ll}0 & -1 \\
-1 & 0\end{array}\right),\quad \chi_{x}\mapsto \left(\begin{array}{ll}0 & 1 \\
 1&0\end{array}\right),~\chi_{y}\mapsto \left(\begin{array}{ll}-1 & 0 \\
 0&1\end{array}\right).
\end{align*}

\item[(9)]  $Y_{x^{3},1}:~$ Let $0\leq k\leq 3$ and $0\leq l\leq 1$, we define $\zeta_{x^{k}y^{l}}\mapsto 0$,
\begin{align*}
e_{x^{k}y^{l}}\mapsto \left(\begin{array}{ll}\delta_{x^{k}y^{l},x^{3}} & 0 \\
0 & \delta_{x^{k}y^{l},x^{3}y}\end{array}\right),~t\mapsto \left(\begin{array}{ll}0 & \xi \\
\xi & 0\end{array}\right),~\chi_{x}\mapsto \left(\begin{array}{ll}0 & -1 \\
 1&0\end{array}\right),~\chi_{y}\mapsto \left(\begin{array}{ll}1 & 0 \\
 0&-1\end{array}\right).
\end{align*}

\item[(10)]$Y_{x^{3},2}:$ Let $0\leq k\leq 3$ and $0\leq l\leq 1$, we define $\zeta_{x^{k}y^{l}}\mapsto 0$,
\begin{align*}
e_{x^{k}y^{l}}\mapsto \left(\begin{array}{ll}\delta_{x^{k}y^{l},x^{3}} & 0 \\
0 & \delta_{x^{k}y^{l},x^{3}y}\end{array}\right),~t\mapsto \left(\begin{array}{ll}0 & -\xi \\
-\xi & 0\end{array}\right),\quad \chi_{x}\mapsto \left(\begin{array}{ll}0 & -1 \\
 1&0\end{array}\right),~\chi_{y}\mapsto \left(\begin{array}{ll}1 & 0 \\
 0&-1\end{array}\right).
\end{align*}

\item[(11)]  $Y_{x^{3},3}:~$ Let $0\leq k\leq 3$ and $0\leq l\leq 1$, we define $\zeta_{x^{k}y^{l}}\mapsto 0$,
\begin{align*}
e_{x^{k}y^{l}}\mapsto \left(\begin{array}{ll}\delta_{x^{k}y^{l},x^{3}} & 0 \\
0 & \delta_{x^{k}y^{l},x^{3}y}\end{array}\right),~t\mapsto \left(\begin{array}{ll}0 & \xi \\
\xi & 0\end{array}\right),~\chi_{x}\mapsto \left(\begin{array}{ll}0 & 1 \\
 1&0\end{array}\right),~\chi_{y}\mapsto \left(\begin{array}{ll}-1 & 0 \\
 0&1\end{array}\right).
\end{align*}

\item[(12)]$Y_{x^{3},4}:$ Let $0\leq k\leq 3$ and $0\leq l\leq 1$, we define $\zeta_{x^{k}y^{l}}\mapsto 0$,
\begin{align*}
e_{x^{k}y^{l}}\mapsto \left(\begin{array}{ll}\delta_{x^{k}y^{l},x^{3}} & 0 \\
0 & \delta_{x^{k}y^{l},x^{3}y}\end{array}\right),~t\mapsto \left(\begin{array}{ll}0 & -\xi \\
-\xi & 0\end{array}\right),\quad \chi_{x}\mapsto \left(\begin{array}{ll}0 & 1 \\
 1&0\end{array}\right),~\chi_{y}\mapsto \left(\begin{array}{ll}-1 & 0 \\
 0&1\end{array}\right).
\end{align*}
\item[(13)] $Y_{x^{3},5}:~$ Let $0\leq k\leq 3$ and $0\leq l\leq 1$, we define $\zeta_{x^{k}y^{l}}\mapsto 0$,
\begin{align*}
e_{x^{k}y^{l}}\mapsto \left(\begin{array}{ll}\delta_{x^{k}y^{l},x^{3}} & 0 \\
0 & \delta_{x^{k}y^{l},x^{3}y}\end{array}\right),~t\mapsto \left(\begin{array}{ll}0 & -1 \\
1 & 0\end{array}\right),~\chi_{x}\mapsto \left(\begin{array}{ll}0 & 1 \\
 1&0\end{array}\right),~\chi_{y}\mapsto \left(\begin{array}{ll}1 & 0 \\
 0&-1\end{array}\right).
\end{align*}

\item[(14)]$Y_{x^{3},6}:$ Let $0\leq k\leq 3$ and $0\leq l\leq 1$, we define $\zeta_{x^{k}y^{l}}\mapsto 0$,
\begin{align*}
e_{x^{k}y^{l}}\mapsto \left(\begin{array}{ll}\delta_{x^{k}y^{l},x^{3}} & 0 \\
0 & \delta_{x^{k}y^{l},x^{3}y}\end{array}\right),~t\mapsto \left(\begin{array}{ll}0 & 1 \\
-1 & 0\end{array}\right),\quad \chi_{x}\mapsto \left(\begin{array}{ll}0 & 1 \\
 1&0\end{array}\right),~\chi_{y}\mapsto \left(\begin{array}{ll}1 & 0 \\
 0&-1\end{array}\right).
\end{align*}
\item[(15)] $Y_{x^{3},7}:$ Let $0\leq k\leq 3$ and $0\leq l\leq 1$, we define $\zeta_{x^{k}y^{l}}\mapsto 0$,
\begin{align*}
e_{x^{k}y^{l}}\mapsto \left(\begin{array}{ll}\delta_{x^{k}y^{l},x^{3}} & 0 \\
0 & \delta_{x^{k}y^{l},x^{3}y}\end{array}\right),~t\mapsto \left(\begin{array}{ll}0 & -1 \\
1 & 0\end{array}\right),~\chi_{x}\mapsto \left(\begin{array}{ll}0 & -1 \\
 1&0\end{array}\right),~\chi_{y}\mapsto \left(\begin{array}{ll}-1 & 0 \\
 0&1\end{array}\right).
\end{align*}

\item[(16)] $Y_{x^{3},8}:$ Let $0\leq k\leq 3$ and $0\leq l\leq 1$, we define $\zeta_{x^{k}y^{l}}\mapsto 0$,
\begin{align*}
e_{x^{k}y^{l}}\mapsto \left(\begin{array}{ll}\delta_{x^{k}y^{l},x^{3}} & 0 \\
0 & \delta_{x^{k}y^{l},x^{3}y}\end{array}\right),~t\mapsto \left(\begin{array}{ll}0 & 1 \\
-1 & 0\end{array}\right),\quad \chi_{x}\mapsto \left(\begin{array}{ll}0 & -1 \\
 1&0\end{array}\right),~\chi_{y}\mapsto \left(\begin{array}{ll}-1 & 0 \\
 0&1\end{array}\right).
\end{align*}
\end{enumerate}
\begin{enumerate}
\item[(S3)]~We denote them through a unified way $Y_{\mu,j}$ for $\mu\in \{x^{},~x^{3}\},\;0\leq j\leq 8.$
\end{enumerate}

Finally, we construct the following 16 two-dimensional spaces $\k \langle v,~w\rangle$ together with the corresponding $\mathcal{D}(H)$-action:

\begin{enumerate}
\item[(1)] $Z_{1,1}:~$ Let $0\leq k\leq 3$ and $0\leq l\leq 1$, we define $\zeta_{x^{k}y^{l}}\mapsto 0$,
\begin{align*}
e_{x^{k}y^{l}}\mapsto \left(\begin{array}{ll}\delta_{x^{k}y^{l},1} & 0 \\
0 & \delta_{x^{k}y^{l},y}\end{array}\right),~t\mapsto \left(\begin{array}{ll}1 & 0 \\
0 & 1\end{array}\right),~\chi_{x}\mapsto \left(\begin{array}{ll}0 & 1 \\
 1&0\end{array}\right),~\chi_{y}\mapsto \left(\begin{array}{ll}1 & 0 \\
 0&-1\end{array}\right).
\end{align*}

\item[(2)] $Z_{1,2}:~$ Let $0\leq k\leq 3$ and $0\leq l\leq 1$, we define $\zeta_{x^{k}y^{l}}\mapsto 0$,
\begin{align*}
e_{x^{k}y^{l}}\mapsto \left(\begin{array}{ll}\delta_{x^{k}y^{l},1} & 0 \\
0 & \delta_{x^{k}y^{l},y}\end{array}\right),~t\mapsto \left(\begin{array}{ll}-1 & 0 \\
0 & -1\end{array}\right),~\chi_{x}\mapsto \left(\begin{array}{ll}0 & 1 \\
 1&0\end{array}\right),~\chi_{y}\mapsto \left(\begin{array}{ll}1 & 0 \\
 0&-1\end{array}\right).
\end{align*}

\item[(3)] $Z_{1,3}:~$ Let $0\leq k\leq 3$ and $0\leq l\leq 1$, we define $\zeta_{x^{k}y^{l}}\mapsto 0$,
\begin{align*}
e_{x^{k}y^{l}}\mapsto \left(\begin{array}{ll}\delta_{x^{k}y^{l},1} & 0 \\
0 & \delta_{x^{k}y^{l},y}\end{array}\right),~t\mapsto \left(\begin{array}{ll}1 & 0 \\
0 & 1\end{array}\right),~\chi_{x}\mapsto \left(\begin{array}{ll}0 & -1 \\
 1&0\end{array}\right),~\chi_{y}\mapsto \left(\begin{array}{ll}1 & 0 \\
 0&-1\end{array}\right).
\end{align*}

\item[(4)] $Z_{1,4}:~$ Let $0\leq k\leq 3$ and $0\leq l\leq 1$, we define $\zeta_{x^{k}y^{l}}\mapsto 0$,
\begin{align*}
e_{x^{k}y^{l}}\mapsto \left(\begin{array}{ll}\delta_{x^{k}y^{l},1} & 0 \\
0 & \delta_{x^{k}y^{l},y}\end{array}\right),~t\mapsto \left(\begin{array}{ll}-1 & 0 \\
0 & -1\end{array}\right),~\chi_{x}\mapsto \left(\begin{array}{ll}0 & -1 \\
 1&0\end{array}\right),~\chi_{y}\mapsto \left(\begin{array}{ll}1 & 0 \\
 0&-1\end{array}\right).
\end{align*}

\item[(5)] $Z_{1,5}:~$ Let $0\leq k\leq 3$ and $0\leq l\leq 1$, we define $\zeta_{x^{k}y^{l}}\mapsto 0$,
\begin{align*}
e_{x^{k}y^{l}}\mapsto \left(\begin{array}{ll}\delta_{x^{k}y^{l},1} & 0 \\
0 & \delta_{x^{k}y^{l},y}\end{array}\right),~t\mapsto \left(\begin{array}{ll}1 & 0 \\
0 & -1\end{array}\right),~\chi_{x}\mapsto \left(\begin{array}{ll}0 & 1 \\
 1&0\end{array}\right),~\chi_{y}\mapsto \left(\begin{array}{ll}-1 & 0 \\
 0&1\end{array}\right).
\end{align*}

\item[(6)] $Z_{1,6}:~$ Let $0\leq k\leq 3$ and $0\leq l\leq 1$, we define $\zeta_{x^{k}y^{l}}\mapsto 0$,
\begin{align*}
e_{x^{k}y^{l}}\mapsto \left(\begin{array}{ll}\delta_{x^{k}y^{l},1} & 0 \\
0 & \delta_{x^{k}y^{l},y}\end{array}\right),~t\mapsto \left(\begin{array}{ll}-1 & 0 \\
0 & 1\end{array}\right),~\chi_{x}\mapsto \left(\begin{array}{ll}0 & 1 \\
 1&0\end{array}\right),~\chi_{y}\mapsto \left(\begin{array}{ll}-1 & 0 \\
 0&1\end{array}\right).
\end{align*}

\item[(7)] $Z_{1,7}:~$ Let $0\leq k\leq 3$ and $0\leq l\leq 1$, we define $\zeta_{x^{k}y^{l}}\mapsto 0$,
\begin{align*}
e_{x^{k}y^{l}}\mapsto \left(\begin{array}{ll}\delta_{x^{k}y^{l},1} & 0 \\
0 & \delta_{x^{k}y^{l},y}\end{array}\right),~t\mapsto \left(\begin{array}{ll}1 & 0 \\
0 & -1\end{array}\right),~\chi_{x}\mapsto \left(\begin{array}{ll}0 & -1 \\
 1&0\end{array}\right),~\chi_{y}\mapsto \left(\begin{array}{ll}-1 & 0 \\
 0&1\end{array}\right).
\end{align*}

\item[(8)] $Z_{1,8}:~$ Let $0\leq k\leq 3$ and $0\leq l\leq 1$, we define $\zeta_{x^{k}y^{l}}\mapsto 0$,
\begin{align*}
e_{x^{k}y^{l}}\mapsto \left(\begin{array}{ll}\delta_{x^{k}y^{l},1} & 0 \\
0 & \delta_{x^{k}y^{l},y}\end{array}\right),~t\mapsto \left(\begin{array}{ll}-1 & 0 \\
0 & 1\end{array}\right),~\chi_{x}\mapsto \left(\begin{array}{ll}0 & -1 \\
 1&0\end{array}\right),~\chi_{y}\mapsto \left(\begin{array}{ll}-1 & 0 \\
 0&1\end{array}\right).
\end{align*}

\item[(9)] $Z_{x^{2},1}:~$ Let $0\leq k\leq 3$ and $0\leq l\leq 1$, we define $\zeta_{x^{k}y^{l}}\mapsto 0$,
\begin{align*}
e_{x^{k}y^{l}}\mapsto \left(\begin{array}{ll}\delta_{x^{k}y^{l},x^{2}} & 0 \\
0 & \delta_{x^{k}y^{l},x^{2}y}\end{array}\right),~t\mapsto \left(\begin{array}{ll}\xi & 0 \\
0 & \xi\end{array}\right),~\chi_{x}\mapsto \left(\begin{array}{ll}0 & 1 \\
 1&0\end{array}\right),~\chi_{y}\mapsto \left(\begin{array}{ll}1 & 0 \\
 0&-1\end{array}\right).
\end{align*}

\item[(10)] $Z_{x^{2},2}:~$ Let $0\leq k\leq 3$ and $0\leq l\leq 1$, we define $\zeta_{x^{k}y^{l}}\mapsto 0$,
\begin{align*}
e_{x^{k}y^{l}}\mapsto \left(\begin{array}{ll}\delta_{x^{k}y^{l},x^{2}} & 0 \\
0 & \delta_{x^{k}y^{l},x^{2}y}\end{array}\right),~t\mapsto \left(\begin{array}{ll}-\xi & 0 \\
0 & -\xi\end{array}\right),~\chi_{x}\mapsto \left(\begin{array}{ll}0 & 1 \\
 1&0\end{array}\right),~\chi_{y}\mapsto \left(\begin{array}{ll}1 & 0 \\
 0&-1\end{array}\right).
\end{align*}

\item[(11)] $Z_{x^{2},3}:~$ Let $0\leq k\leq 3$ and $0\leq l\leq 1$, we define $\zeta_{x^{k}y^{l}}\mapsto 0$,
\begin{align*}
e_{x^{k}y^{l}}\mapsto \left(\begin{array}{ll}\delta_{x^{k}y^{l},x^{2}} & 0 \\
0 & \delta_{x^{k}y^{l},x^{2}y}\end{array}\right),~t\mapsto \left(\begin{array}{ll}\xi & 0 \\
0 & \xi\end{array}\right),~\chi_{x}\mapsto \left(\begin{array}{ll}0 & -1 \\
 1&0\end{array}\right),~\chi_{y}\mapsto \left(\begin{array}{ll}1 & 0 \\
 0&-1\end{array}\right).
\end{align*}

\item[(12)] $Z_{x^{2},4}:~$ Let $0\leq k\leq 3$ and $0\leq l\leq 1$, we define $\zeta_{x^{k}y^{l}}\mapsto 0$,
\begin{align*}
e_{x^{k}y^{l}}\mapsto \left(\begin{array}{ll}\delta_{x^{k}y^{l},x^{2}} & 0 \\
0 & \delta_{x^{k}y^{l},x^{2}y}\end{array}\right),~t\mapsto \left(\begin{array}{ll}-\xi & 0 \\
0 & -\xi\end{array}\right),~\chi_{x}\mapsto \left(\begin{array}{ll}0 & -1 \\
 1&0\end{array}\right),~\chi_{y}\mapsto \left(\begin{array}{ll}1 & 0 \\
 0&-1\end{array}\right).
\end{align*}

\item[(13)] $Z_{x^{2},5}:~$ Let $0\leq k\leq 3$ and $0\leq l\leq 1$, we define $\zeta_{x^{k}y^{l}}\mapsto 0$,
\begin{align*}
e_{x^{k}y^{l}}\mapsto \left(\begin{array}{ll}\delta_{x^{k}y^{l},x^{2}} & 0 \\
0 & \delta_{x^{k}y^{l},x^{2}y}\end{array}\right),~t\mapsto \left(\begin{array}{ll}\xi & 0 \\
0 & -\xi\end{array}\right),~\chi_{x}\mapsto \left(\begin{array}{ll}0 & 1 \\
 1&0\end{array}\right),~\chi_{y}\mapsto \left(\begin{array}{ll}-1 & 0 \\
 0&1\end{array}\right).
\end{align*}

\item[(14)] $Z_{x^{2},6}:~$ Let $0\leq k\leq 3$ and $0\leq l\leq 1$, we define $\zeta_{x^{k}y^{l}}\mapsto 0$,
\begin{align*}
e_{x^{k}y^{l}}\mapsto \left(\begin{array}{ll}\delta_{x^{k}y^{l},x^{2}} & 0 \\
0 & \delta_{x^{k}y^{l},x^{2}y}\end{array}\right),~t\mapsto \left(\begin{array}{ll}-\xi & 0 \\
0 & \xi\end{array}\right),~\chi_{x}\mapsto \left(\begin{array}{ll}0 & 1 \\
 1&0\end{array}\right),~\chi_{y}\mapsto \left(\begin{array}{ll}-1 & 0 \\
 0&1\end{array}\right).
\end{align*}

\item[(15)] $Z_{x^{2},7}:~$ Let $0\leq k\leq 3$ and $0\leq l\leq 1$, we define $\zeta_{x^{k}y^{l}}\mapsto 0$,
\begin{align*}
e_{x^{k}y^{l}}\mapsto \left(\begin{array}{ll}\delta_{x^{k}y^{l},x^{2}} & 0 \\
0 & \delta_{x^{k}y^{l},x^{2}y}\end{array}\right),~t\mapsto \left(\begin{array}{ll}\xi & 0 \\
0 & -\xi\end{array}\right),~\chi_{x}\mapsto \left(\begin{array}{ll}0 & -1 \\
 1&0\end{array}\right),~\chi_{y}\mapsto \left(\begin{array}{ll}-1 & 0 \\
 0&1\end{array}\right).
\end{align*}

\item[(16)] $Z_{x^{2},8}:~$ Let $0\leq k\leq 3$ and $0\leq l\leq 1$, we define $\zeta_{x^{k}y^{l}}\mapsto 0$,
\begin{align*}
e_{x^{k}y^{l}}\mapsto \left(\begin{array}{ll}\delta_{x^{k}y^{l},x^{2}} & 0 \\
0 & \delta_{x^{k}y^{l},x^{2}y}\end{array}\right),~t\mapsto \left(\begin{array}{ll}-\xi & 0 \\
0 & \xi\end{array}\right),~\chi_{x}\mapsto \left(\begin{array}{ll}0 & -1 \\
 1&0\end{array}\right),~\chi_{y}\mapsto \left(\begin{array}{ll}-1 & 0 \\
 0&1\end{array}\right).
\end{align*}
\end{enumerate}
\begin{enumerate}
\item[(S4)]~We denote them through a unified way $Z_{\mu,j}$ for $\mu\in \{1,~x^{2}\},\;0\leq j\leq 8.$
\end{enumerate}

\begin{lemma}\label{lemma4.20}
\begin{enumerate}
\item Above $56$ two-dimensional spaces~$(S1\mbox{-}S4)$ together with the corresponding action are $\mathcal{D}(H)$\mbox{-}modules.
\item The $\mathcal{D}(H)$\mbox{-}modules $(S1\mbox{-}S4)$ are irreducible.

\end{enumerate}
\end{lemma}
\begin{proof}
\begin{enumerate}
\item We only verify the case $(S3)$ since the other cases can be proved similarly. Let $\mu\in \{x^{},~x^{3}\}$. Suppose ~$v\in\chi_{1}V_{\mu}$,~$w=\chi_{x}\cdot v$ and $V=\Bbbk\langle v,w\rangle$. Since $e_{g}\chi_{h}=\chi_{h}e_{gh(h^{-1}\triangleleft t)}$, it follows that $w\in\chi_{1}V_{\mu y}$. Since $t\cdot v_{\mu}\in \chi_{1} V_{\mu y}$, we can suppose $t\cdot v=\lambda w$. Let $A_{S}=\Bbbk\langle \chi_{s}\rangle_{s\in S}$ be an algebra generated by $\{\chi_{s}\mid s\in S\}$. Due to~$(\ref{3-3})$, it follows that $A_{S}$ is a commutative algebra. Since every  simple module over $A_{S}$ is $1$\mbox{-}dimensional, we can assume
\begin{equation} \label{equation4.011}
\chi_{x^{2}}v=\widetilde{m} v,\quad \chi_{y}v=\widetilde{n} v,
\end{equation}
 where $\widetilde{m},\widetilde{n}\in \Bbbk$.  By relation~(\ref{3-3}), it follows that
\begin{align}
(\chi_{x^{2}})^{2}=\chi_{x^{4}}=\chi_{1},\quad\quad (\chi_{y})^{2}=\chi_{y^{2}}=\chi_{1}.
\end{align}
Thus, we have
 \begin{equation} \label{equation4.0}
 v\stackrel{}{=}\chi_{1}v\stackrel{(4.7)}{=}(\chi_{x^{4}})v\stackrel{(4.7)}{=}(\chi_{x^{2}})^{2}v=\widetilde{m}^{2}v,\quad v\stackrel{}{=}\chi_{1}v\stackrel{(4.7)}{=}\chi_{y^{2}}v\stackrel{(4.7)}{=}(\chi_{y})^{2}v=\widetilde{n}^{2}v.
 \end{equation}

 Let $\chi_{s}v=\beta(s)v$ for $s\in S$ defined in~$(\ref{equation4.011})$. Then $(\beta(x^{2}))^{2}=\beta(x^{4})=1,~(\beta(y))^{2}=\beta(y^{2})=1.$ Since $S$ is the subgroup of $G$ and $T=Sx$, we have following actions of $\mathcal{D}(H)$:
\begin{equation}\label{4-12-0}
\begin{array}{l}
\chi_{sx}(v)=\tau(x,~s, t)\beta(s)w=\beta(s)w,\\
\chi_{sx}(w)=\tau(s,~x, t)\beta(s)\beta(x^{2})v,\\
\chi_{s}w=\beta(s)\tau(s,x,t)w,\chi_{s}v=\beta(s)v,\\
t\cdot w=t\cdot \chi_{x}v=\lambda\tau(x\mu y, x^{3}y,t)\beta(y)\beta(x^{2})v,\\
t^{2}v=\lambda^{2}\tau(x\mu y, x^{3}y,t)\beta(y)\beta(x^{2})v=\sigma(\mu,t,t)v,\\
\lambda^{2}\tau(x\mu y, x^{3}y,t)\beta(y)\beta(x^{2})=\sigma(\mu,t,t).\\
\end{array}
\end{equation}
Due to~$(\ref{equation4.0})$ and~(\ref{4-12-0}), it follows that
\begin{equation}\label{4.140}
\widetilde{m}=\beta(x^{2})=\pm1,\quad \widetilde{n}=\beta(y)=\pm 1,\quad\lambda^{2}=\frac{\sigma(\mu,t,t)}{-\beta(y)\beta(x^{2})}.
\end{equation}
i.e. $\lambda$ is determined by $\beta(x^{2})$ and $\beta(y)$. Thus, there are 16 two-dimensional left-$\mathcal{D}(H)$modules as $(S3)$.

\item
\begin{enumerate}
\item[(a)] If $\Bbbk\langle v+\lambda w\rangle$ is a $\mathcal{D}(H)\mbox{-}$submodule of $(S1)$, then $\zeta_{x^{k}y^{l}}(v+\lambda w)~(0\leq k\leq 3,~0\leq l\leq 1)$ and $t(v+\lambda w)$ belong to $\Bbbk\langle v+\lambda w\rangle$. But $$\zeta_{y}(v+\lambda w)=v-\lambda w~\text{or}~-v+\lambda w\notin\Bbbk\langle v+\lambda w\rangle,$$ this leads a contradiction. Thus, there are 16 non-isomorphic two-dimensional simple $\mathcal{D}(H)\mbox{-}$modules as $(S1)$.
  \item[(b)]$(S2)$ follow in a similar manner as $(a)$.
  \item[(c)]If $\Bbbk\langle v+\lambda w\rangle$ is a $\mathcal{D}(H)\mbox{-}$submodule of $(S3)$, then $\chi_{x^{k}y^{l}}(v+\lambda w)$ must belong to $\Bbbk\langle v+\lambda w\rangle~(0\leq k\leq 3,~0\leq l\leq 1).$ But $\chi_{y}(v+\lambda w)\notin \Bbbk\langle v+\lambda w\rangle$ by directly computation. Thus, there are 16 non-isomorphic two-dimensional simple $\mathcal{D}(H)\mbox{-}$modules, as shown in the $(S3)$.
\item[(d)]The $\mathcal{D}(H)\mbox{-}$modules in $(S4)$ are simple by the same arguments as $(c)$.
\end{enumerate}
\end{enumerate}
\end{proof}

With these preparations, we can give the first main result.
\begin{theorem}\label{theorem4.7}
There are exactly 88 non-isomorphic simple Yetter-Drinfel'd modules over $H$, and 32 of them are one-dimensional and the other 56 ones are two-dimensional.
\begin{center}
 \emph{\texttt{TABLE $1$. Simple Yetter-Drinfel'd modules over}} $H$
\scalebox{0.9}{
\begin{tabular}{r|l}
\hline
(S0)& $V_{1,0}^{+}$,~$V_{1,1}^{+}$,~$V_{1,2}^{+}$,~$V_{1,3}^{+}$,~$V_{1,0}^{-}$,~$V_{1,1}^{-}$,~$V_{1,2}^{-}$,~$V_{1,3}^{-}$\\

& $V_{y,0}^{+}$,~$V_{y,1}^{+}$,~$V_{y,2}^{+}$,~$V_{y,3}^{+}$,~$V_{1,0}^{-}$,~$V_{y,1}^{-}$,~$V_{y,2}^{-}$,~$V_{y,3}^{-}$\\

& $V_{x^{2},0}^{+}$,~$V_{x^{2},1}^{+}$,~$V_{x^{2},2}^{+}$,~$V_{x^{2},3}^{+}$,~$V_{x^{2},0}^{-}$,~$V_{x^{2},1}^{-}$,~$V_{x^{2},2}^{-}$,~$V_{x^{2},3}^{-}$\\

&$V_{x^{2}y,0}^{+}$,~$V_{x^{2}y,1}^{+}$,~$V_{x^{2}y,2}^{+}$,~$V_{x^{2}y,3}^{+}$,~$V_{x^{2}y,0}^{-}$,~$V_{x^{2}y,1}^{-}$,~$V_{x^{2}y,2}^{-}$,~$V_{x^{2}y,3}^{-}$\\
\hline
(S1)& $W_{x,0,-}^{+}$,~$W_{x,1,-}^{+}$,~$W_{x,2,-}^{+}$,~$W_{x,3,-}^{+}$,  $W_{x,0,+}^{+}$,~$W_{x,1,+}^{+}$,~$W_{x,2,+}^{+}$,~$W_{x,3,+}^{+}$\\

& $W_{x^{3},0}^{-}$,~$W_{x^{3},1}^{-}$,~$W_{x^{3},2}^{-}$,~$W_{x^{3},3}^{-}$ , $W_{x^{3},0}^{-}$,~$W_{x^{3},1}^{-}$,~$W_{x^{3},2}^{-}$,~$W_{x^{3},3}^{-}$\\
\hline
$(S2)$& $X_{1,0}^{+}$,~$X_{1,1}^{+}$,~$X_{y,0}^{+}$,~$X_{y,1}^{+}$,~$X_{x^{2},0}^{-}$,~$X_{x^{2},1}^{-}$,~$X_{x^{2}y,0}^{-}$,~$X_{x^{2}y,1}^{-}$\\
\hline

$(S3)$& $Y_{x,1}$,~$Y_{x,2}$,~$Y_{x,3}$,~$Y_{x,4}$,~$Y_{x,5}$,~$Y_{x,6}$,~$Y_{x,7}$,~$Y_{x,8}$,\\

& $Y_{x^{3},1}$,~$Y_{x^{3},2}$,~$Y_{x^{3},3}$,~$Y_{x^{3},4}$,~$Y_{x^{3},5}$,~$Y_{x^{3},6}$,~$Y_{x^{3},7}$,~$Y_{x^{3},8}$,\\
\hline
$(S4)$& $Z_{1,1}$,~$Z_{1,2}$,~$Z_{1,3}$,~$Z_{1,4}$,~$Z_{1,5}$,~$Z_{1,6}$,~$Z_{1,7}$,~$Z_{1,8}$,\\
& $Z_{x^{2},1}$,~$Z_{x^{2},2}$,~$Z_{x^{2},3}$,~$Z_{x^{2},4}$,~$Z_{x^{2},5}$,~$Z_{x^{2},6}$,~$Z_{x^{2},7}$,~$Z_{x^{2},8}$,\\
\hline
\end{tabular}}
\end{center}
\begin{proof}
By Lemmas~\ref{lemma4.1},~\ref{lemma4.20}, there are 32 non-isomorphic one-dimensional and 56 non-isomorphic two-dimensional simple Yetter-Drinfel'd modules over $H$. Due to Lemma~\ref{lemma2.3} and $32\cdot 1^{2}+56\cdot 2^{2}=256={\rm dim}\mathcal{D}(H)$, we finish the proof.
 \end{proof}
\end{theorem}

\subsection{Left-left Yetter-Drinfel'd modules}

\begin{lemma}
Fix a basis of each module in $(S0\mbox{-}S4)$, the following table give us braidings of $\mathcal{D}(H)\mbox{-}$modules~$(S0\mbox{-}S4)$ as left-left Yetter-Drinfel'd modules over $H$.

\begin{center}
 \emph{\texttt{TABLE $2$. Braidings of $\mathcal{D}(H)\mbox{-}$modules in~$(S0,S1,S2,S4)$}}
\scalebox{0.9}{
\begin{tabular}{r|l}
\hline
$(S0)$ & $c(v\otimes v)=\pm (v\otimes v)$\\

$(S1)$ & $q=\left(\begin{array}{ll}\xi^{i}& \xi^{i}  \\
-\xi^{i} & \xi^{i}\end{array}\right),~~~\left(\begin{array}{ll}\xi^{i}& -\xi^{i}  \\
\xi^{i} & \xi^{i}\end{array}\right)~(i=0,~1,~2,~3)$\\
$(S2)$ & $q=\left(\begin{array}{ll}1& 1  \\
1 & 1\end{array}\right),~~~\left(\begin{array}{ll}-1& -1  \\
-1 & -1\end{array}\right)$\\
$(S4)$ &$q=\left(\begin{array}{ll}\pm 1& \mp1  \\
\pm1 & \pm1\end{array}\right),~~~\left(\begin{array}{ll}\pm \xi& \mp \xi  \\
\pm \xi & \pm \xi\end{array}\right),$ \\
& \\
\hline
\end{tabular}}
\end{center}

\begin{center}
 \emph{\texttt{TABLE $3$. Braidings of $\mathcal{D}(H)\mbox{-}$modules in~$(S3)$}}
\scalebox{0.9}{
\begin{tabular}{|c|c|c|c|c|c|c|c|c|c|c|c|c|}

\hline
&   $c(v\otimes v)={\pm\xi}w\otimes w$ & $c(v\otimes v)=\mp w\otimes w$ & $c(v\otimes v)=\pm w\otimes w$ & $c(v\otimes v)={\mp  \xi}w\otimes w$ \\
 &$c(v\otimes w)={\pm \xi}v\otimes w$ & $c(v\otimes w)=\pm v\otimes w$ & $c(v\otimes w)=\pm v\otimes w$ & $c(v\otimes w)=\pm\xi v\otimes w$  \\
 $\mu=x$ &$c(w\otimes v)=\pm\xi w\otimes v$&$c(w\otimes v)=\pm w\otimes v$&  $c(w\otimes v)=\pm w\otimes v$ &$c(w\otimes v)=\pm{\xi}w\otimes v$  \\
  & $c(w\otimes w)={\mp \xi}v\otimes v$&$c(w\otimes w)=\pm v\otimes v$& $c(w\otimes w)=\mp v\otimes v$& $c(w\otimes w)=\pm\xi v\otimes v$  \\
\hline
& $c(v\otimes v)={\mp }w\otimes w$    & $c(v\otimes v)=\mp\xi w\otimes w$ & $c(v\otimes v)=\mp{\xi} w\otimes w$  & $c(v\otimes v)=\mp w\otimes w$ \\
      &   $c(v\otimes w)=\mp v\otimes w$  & $c(v\otimes w)=\pm {\xi} v\otimes w$ & $c(v\otimes w)=\mp{\xi}v\otimes w$   & $c(v\otimes w)=\pm v\otimes w$  \\
$\mu=x^{3}$ &  $c(w\otimes v)=\mp w\otimes v$   &$c(w\otimes v)=\pm{\xi} w\otimes v$&  $c(w\otimes v)=\mp{\xi}w\otimes v$     &$c(w\otimes v)=\pm w\otimes v$  \\
  &       $c(w\otimes w)=\pm v\otimes v$    &$c(w\otimes w)=\pm{\xi}v\otimes v$& $c(w\otimes w)=\pm{ \xi}v\otimes v$    & $c(w\otimes w)=\pm v\otimes v$  \\
\hline

\end{tabular}}
\end{center}

\begin{proof}

We only verify the case $(S3)$. By the action of $\mathcal{D}(H)$, we can give following $H$\mbox{-}comodule structure. Then we determine the braidings $c_{V,W}$ for $V,W\in { }_{H}^{H} \mathcal{Y}\mathcal{D}$.
\begin{equation*}
\begin{array}{l}
\rho(v)=v\otimes\Sigma_{s\in S}\beta(s)(e_{s} t)+w\otimes \Sigma_{s\in S}\beta(s)(e_{xs} t),\\
\rho(w)=w\otimes\Sigma_{s\in S}\beta(s)\tau(s,x,t)(e_{s} t)+v\otimes \Sigma_{s\in S}\beta(s)\beta(x^{2})\tau(s,x,t)S(e_{xs} t),\\
\rho_{l}(v)=\Sigma_{s\in S}\beta(s)S(e_{s} t)\otimes v+ \Sigma_{s\in S}\beta(s)S(e_{xs} t)\otimes w,\\
\rho_{l}(w)=\Sigma_{s\in S}\beta(s)\tau(s,x,t)S(e_{s} t)\otimes w+\Sigma_{s\in S}\beta(s)\beta(x^{2})\tau(s,x,t)S(e_{xs} t)\otimes v,\\
c(v\otimes v)=\Sigma_{s\in S}[\beta(s)\mathcal{S}(e_{s} t)\cdot v\otimes v]+ \Sigma_{s\in S}[\beta(s)\mathcal{S}(e_{xs} t)\cdot v\otimes w]=\Sigma_{s\in S}[\beta(s)\\
\mathcal{S}(e_{xs} t)\cdot v\otimes w],\\
c(v\otimes w)=\Sigma_{s\in S}[\beta(s)\mathcal{S}(e_{s} t)\cdot w\otimes v]+ \Sigma_{s\in S}[\beta(s)\mathcal{S}(e_{xs} t)\cdot w\otimes w]=\Sigma_{s\in S}[\beta(s)\\
\mathcal{S}(e_{xs} t)\cdot w\otimes w],\\
c(w\otimes v)=\Sigma_{s\in S}[\beta(s)\tau(s,x,t)\mathcal{S}(e_{s} t)\cdot v\otimes w]+\Sigma_{s\in S}[\beta(s)\beta(x^{2})\tau(s,x,t)\\
S(e_{xs} t)\cdot v\otimes v]=\Sigma_{s\in S}[\beta(s)\beta(x^{2})\tau(s,x,t)\mathcal{S}(e_{xs} t)\cdot v\otimes v],\\
c(w\otimes w)=\Sigma_{s\in S}[\beta(s)\tau(s,x,t)S(e_{s} t)\cdot w\otimes w]+\Sigma_{s\in S}[\beta(s)\beta(x^{2})\tau(s,x,t),\\
S(e_{xs} t)\cdot w\otimes v] =\Sigma_{s\in S}[\beta(s)\beta(x^{2})\tau(s,x,t)\mathcal{S}(e_{xs} t)\cdot w\otimes v].\\
\end{array}
\end{equation*}
Due to (\ref{2-2}), we have the following formula which will be used to calculate the braiding.
\begin{equation}
\begin{array}{l}
\quad S(e_{1} t)=e_{1} t,\quad \quad\quad\quad\quad \quad S(e_{y} t)=e_{y} t,\quad\quad\quad \quad S(e_{xy} t)=e_{x^{3}} t, \\
S(e_{x^{3}} t)=(e_{xy} t), \quad\quad\quad \quad\quad S(e_{x} t)=-e_{x^{3}y} t, \quad\quad S(e_{x^{2}} t)=-e_{x^{2}} t,\\
S(e_{x^{2}y} t)=-e_{x^{2}y} t,\quad\quad\quad\quad S(e_{x^{3}y} t)=-(e_{x} t).\\
\end{array}
\end{equation}
Let $\mu=x$, we get the following braiding:
\begin{equation*}
\begin{array}{l}
c(v\otimes v)=(-e_{x^{3}y}t+\beta(y)e_{x^{3}}t+\beta(x^{2})e_{xy}t-\beta(x^{2}y)e_{x}t)\cdot v\otimes w=\lambda\beta(x^{2})w\otimes w,\\
c(v\otimes w)=(-e_{x^{3}y}t+\beta(y)e_{x^{3}}t+\beta(x^{2})e_{xy}t-\beta(x^{2}y)e_{x}t)\cdot w\otimes w=\lambda v\otimes w,\\
c(w\otimes v)=(-\beta(x^{2})e_{x^{3}y}t-\beta(x^{2}y)e_{x^{3}}t+e_{xy}t+\beta(y)e_{x}t)\cdot v\otimes v=\lambda w\otimes v,\\
c(w\otimes w)=(-\beta(x^{2})e_{x^{3}y}t-\beta(x^{2}y)e_{x^{3}}t+e_{xy}t+\beta(y)e_{x}t)\cdot w\otimes v=-\lambda \beta(x^{2})v\otimes v.\\
\end{array}
\end{equation*}
Let $\mu=x^{3}$, we get the following braiding:
\begin{equation*}
\begin{array}{l}
c(v\otimes v)=(-e_{x^{3}y}t+\beta(y)e_{x^{3}}t+\beta(x^{2})e_{xy}t-\beta(x^{2}y)e_{x}t)\cdot v\otimes w=-\lambda w\otimes w,\\
c(v\otimes w)=(-e_{x^{3}y}t+\beta(y)e_{x^{3}}t+\beta(x^{2})e_{xy}t-\beta(x^{2}y)e_{x}t)\cdot w\otimes w=-\lambda \beta(x^{2})v\otimes w,\\
c(w\otimes v)=(-\beta(x^{2})e_{x^{3}y}t-\beta(x^{2}y)e_{x^{3}}t+e_{xy}t+\beta(y)e_{x}t)\cdot v\otimes v=-\lambda \beta(x^{2}) w\otimes v,\\
c(w\otimes w)=(-\beta(x^{2})e_{x^{3}y}t-\beta(x^{2}y)e_{x^{3}}t+e_{xy}t+\beta(y)e_{x}t)\cdot w\otimes v=\lambda v\otimes v.
\end{array}
\end{equation*}

Thus, we get the braidings of left-left Yetter-Drinfel'd modules over $H$ of table $(S3)$.
\end{proof}
\end{lemma}

\section{Finite Dimensional Nichols algebras of ${ }_{H}^{H} \mathcal{Y} \mathcal{D}$}

In the previous section, we already gave all simple Yetter-Drinfel'd modules $V$ over $H$. In this section, we want to determine all the finite dimensional Nichols algebras of simple Yetter-Drinfel'd modules over $H$.

First, we give the Nichols algebras of Yetter-Drinfel'd modules that appear in $(S0)$.
\begin{lemma}\label{lemma5.1}
The Nichols algebras of one-dimensional simple Yetter-Drinfel'd modules $(V_{\mu,i}^{\pm},c_{\pm1})~(i=0,1,2,3)$ of $H$ are as follows,
\begin{enumerate}
\item $\mathcal{B}((V_{\mu,i}^{\pm},c_{1}))\simeq \Bbbk[v]$, where $c_{1}(v\otimes v)=v\otimes v$.
\item $\mathcal{B}((V_{\mu,i}^{\pm},c_{-1}))\simeq \Bbbk[v]/\langle v^{2}\rangle=\bigwedge\Bbbk \langle v \rangle$, and ${\rm dim}(\mathcal{B}((V_{\mu,i}^{\pm},c_{-1})))=2$, where $c_{-1}(v\otimes v)=-v\otimes v.$
\end{enumerate}
\begin{proof}
It is well-known that this claim holds. For~$(2)$, the Nichols algebra is of type $A_{1}$.
\end{proof}
\end{lemma}

Second, we give the Nichols algebras of Yetter-Drinfel'd modules that appear in ~$(S1)$. Denote by $(W_{\mu,i,\pm}^{\pm},{c_{i}})~(i=0,1,2,3)$ the simple Yetter-Drinfel'd modules of $H$ with braiding matrices
$$q_{1}=\left(\begin{array}{ll}\xi^{i}& \xi^{i}  \\
-\xi^{i} & \xi^{i}\end{array}\right),~~~q_{2}=\left(\begin{array}{ll}\xi^{i}& -\xi^{i}  \\
\xi^{i} & \xi^{i}\end{array}\right).$$
\begin{lemma}\label{lemma5.2}
Then Nichlos algebras of $(W_{\mu,i,\pm}^{\pm},c_{i})~(i=0,1,2,3)$ are as follows,
\begin{enumerate}
\item If $i=1~\text{or}~3,~q=q_{1}$, then $\mathcal{B}((W_{\mu,i,\pm}^{\pm},c_{i}))\simeq T(W_{\mu,i,\pm}^{\pm})/\langle vw\mp \xi^{i} wv,wv\pm \xi^{i} vw,v^{4},w^{4}\rangle$, and ${\rm dim}(\mathcal{B}((W_{\mu,i,\pm}^{\pm},c_{i})))=16$.
\item If  $i=1~\text{or}~3,~q=q_{2}$, then $\mathcal{B}((W_{\mu,i,\pm}^{\pm},c_{i}))\simeq T(W_{\mu,i,\pm}^{\pm})/\langle vw\pm \xi^{i} wv,wv\mp \xi^{i} vw,v^{4},w^{4}\rangle$, and ${\rm dim}(\mathcal{B}((W_{\mu,i,\pm}^{\pm},c_{i})))=16$.
\item If $i=2,~q=q_{1},q_{2}$, then Nichols algebras of $(W_{\mu,i,\pm}^{\pm},c_{i})$ are 8 dimensional and of type $A_{2}$.
\item If $i=0,~q=q_{1},q_{2}$, then $\mathcal{B}((W_{\mu,i,\pm}^{\pm},c_{i}))$ are infinite dimensional.
\end{enumerate}
\begin{proof}

Let $i=1~\text{or}~3$, then we have $q_{ij}q_{ji}=1$ for all $i\neq j,$ and the order of $q_{ii}$ are all equal to 4. By the result of~(\cite{Andru98}, \cite[Example~27]{Andruskiewitsch}), we have
$\mathcal{B}((W_{\mu,i,\pm}^{\pm},c_{j}))\simeq T(W_{\mu,i,\pm}^{\pm})/\langle vw-q_{12}wv,wv-q_{21}v,v^{4},w^{4}\rangle$
Furthermore, $\{v^{a_{1}} w^{a_{2}}: 0 \leq a_{k} \leq 3\}$ is a basis of $\mathcal{B}(W_{\mu,i,\pm}^{\pm})$ and ${\rm dim}(\mathcal{B}((W_{\mu,i,\pm}^{\pm},c_{j})))=\Pi_{i=1,2}N^{i}=16.$ For (1),~(2), Nichols algebras are quantum planes. (3) follows from (\cite[Proposition 2.11]{{Grana}}), Nichols algebras are of $A_{2}$ type. For (4), it is easy to show that $\mathcal{B}((W_{\mu,i,\pm}^{\pm},c_{j}))$ are infinite  dimensional.
\end{proof}
\end{lemma}

Next, we give the Nichols algebras of Yetter-Drinfel'd modules that appear in $(S2)$.
\begin{lemma}\label{lemma5.3}
Let $X_{\mu,i}^{\pm}~(i=0~,1)$ be simple Yetter-Drinfel'd modules of $H$ with braiding matrices
  $$q=\left(\begin{array}{ll}\xi^{i}& \xi^{i} \\
\xi^{i}& \xi^{i}\end{array}\right)~(i=0,~2).$$
Then Nichlos algebras of $(X_{\mu,i}^{\pm},c_{i})$ are as follows
\begin{enumerate}
\item If $i=0$, then $\mathcal{B}((X_{\mu,i}^{\pm},c_{i}))\simeq S(X_{\mu,i}^{\pm})$.
\item If $i=2$, then $\mathcal{B}((X_{\mu,i}^{\pm},c_{i}))\simeq \bigwedge(X_{\mu,i}^{\pm}),$ and $\rm dim(\mathcal{B}((X_{\mu,i}^{\pm},c_{i})))=4.$
\end{enumerate}
\begin{proof}
The claims (1), (2) follow from (\cite[example 31]{Andruskiewitsch}). For (2), the Nichols algebra is of type $A_{1}\times A_{1}$ and $\mathcal{B}((X_{\mu,i}^{\pm},c_{i}))$ is a quantum plane.
\end{proof}
\end{lemma}

We are in a position to give the Nichols algebras of Yetter-Drinfel'd modules that appear in~$(S3)$ now. There are some Nichlos algebras of non-diagonal types which were also studied by Andruskiewitsch~\cite{AG}.
\begin{lemma}\label{theorem5.4}
Let $r,p,m\in \Bbbk$, assume $(Y_{\mu,j}, c_{r,p,m})~(j=0,1,2,3,4,5,6,7,8~, r,p,m\in \Bbbk)$ be simple Yetter-Drinfel'd modules of $H$ with braidings
 \begin{center}
 $c(v\otimes v)=r w\otimes w, c(v\otimes w)=p v\otimes w,$\\
 $c(w\otimes v)=p w\otimes v,c(w\otimes w)=m v\otimes v.$\\
 \end{center}
The Nichlos algebras of $(Y_{\mu,j}, c_{r,p,m})$ are not of diagonal type and given as follow:
\begin{enumerate}
\item If $p=-1, rm=-1$, then
 $$\mathcal{B}((Y_{\mu,j}, c_{r,p,m}))\cong T(Y_{\mu,j})/\langle vw,wv,v^{4}+w^{4}\rangle,~~{\rm dim}(\mathcal{B}((Y_{\mu,j}, c_{r,p,m})))=8.$$
\item  If $p=1$, then ${\rm dim}(\mathcal{B}((Y_{\mu,j}, c_{r,p,m})))=\infty.$
\item If $rm=1,~p=\pm i$ and $r=-i$, then
$$\mathcal{B}((Y_{\mu,j}, c_{r,p,m})) \cong  T(Y_{\mu,j})/ \langle w^{2}+i v^{2}, vwvw,wvwv \rangle,~~{\rm dim}(\mathcal{B}((Y_{\mu,j}, c_{r,p,m})))=16.$$
\item If $rm=1,~p=\pm i,~r=i$, then
$$\mathcal{B}((Y_{\mu,j}, c_{r,p,m})) \cong  T(Y_{\mu,j})/\langle w^{2}-i v^{2}, vwvw,wvwv \rangle,~~{\rm dim}(\mathcal{B}((Y_{\mu,j}, c_{r,p,m})))=16.$$
\end{enumerate}
\end{lemma}
\begin{proof}
  Since $(Y_{\mu,j}, c_{r,p,m})$ is diagonal type if and only if $p^{2}=rm$~(\cite[Remark 3.5]{AG}), it follows that the braidings are not of diagonal type. (1),~(3),~(4) are the results in~(\cite[Proposition 3.16]{AG}). It is sufficient to prove (2), let $n\in \mathbb{N}$ and let
  $$a_n=\underbrace{v\otimes w\otimes v \otimes w \otimes v \otimes w \otimes v\otimes w\cdots }_{n}$$
  be a vector in $\mathcal{B}((Y_{\mu,j}, c_{r,1,m})).$ After direct computation, we have $\Omega_{n}(a_n)=ka_n$ for some $k\not =0\in \Bbbk$. Thus, we have completed the proof.
\end{proof}

Finally, we give the Nichols algebras of Yetter-Drinfel'd modules that appear in $(S4)$.
\begin{lemma}\label{lemma5.5}
Let $(Z_{\mu,j},c_{i})~(j=1,2,3,4,5,6,7,8;~i=0,1,2,3)$ be simple Yetter-Drinfel'd modules over $H$ with braiding matrices
$$q=\left(\begin{array}{ll}\xi^{i}& -\xi^{i}  \\
 \xi^{i} &  \xi^{i}\end{array}\right).$$

Then the Nichols algebras of $(Z_{\mu,j},c_{i})$ are as follows,
\begin{enumerate}
\item If $i=1\text{or}~3$, then $\mathcal{B}((Z_{\mu,j},c_{i}))\simeq T(Z_{\mu,j})/\langle vw\pm \xi^{i} wv,wv\mp \xi^{i}vw,v^{4},w^{4}\rangle$ and
$${\rm dim}(\mathcal{B}((Z_{\mu,j},c_{i})))=16.$$
\item If $i=2$, then $\mathcal{B}((Z_{\mu,j},c_{i}))\simeq T(Z_{\mu,j})/ \langle vw-wv,v^{2},w^{2}\rangle $  and
$${\rm dim}(\mathcal{B}((Z_{\mu,j},c_{i})))=8.$$
\item If $i=0$, then $\mathcal{B}((Z_{\mu,j},c_{i}))$ is infinite dimensional.
\end{enumerate}
\begin{proof}
This claim is similar as ~Lemma~\ref{lemma5.2}.
\end{proof}
\end{lemma}

With these preparations, we are in a position to classify all finite-dimensional Nichols algebras of simple Yetter-Drinfel'd modules $V$ over $H$ now.
 \begin{theorem}\label{theorem5.6}
Let $V$ be the simple Yetter-Drinfel'd modules over $H$, all finite dimensional Nichols algebras of $V$ are as follows,
 \begin{enumerate}
 \item Diagonal type: $A_{1}, A_{2}$ or quantum planes.
 \item Non-diagonal type:  Nichols algebras $\mathcal{B}(V)$ are $8$ or $16$ dimensional.
 \end{enumerate}
 \begin{proof}
 Let $V\in{}{ }_{H}^{H} \mathcal{Y}\mathcal{D}$ be the braided vector spaces of diagonal type, due to Lemmas~\ref{lemma5.1},~\ref{lemma5.2},~\ref{lemma5.3},\\
 \ref{lemma5.5}, it follows that Nichols algebras of $V$ are finite dimensional if and only if $\mathcal{B}(V)$ are of Cartan type $A_{1}, A_{2}$ or quantum planes.
Let $V\in{ }_{H}^{H} \mathcal{Y}\mathcal{D}$ be the braided vector spaces of non-diagonal type, due to Lemma~\ref{theorem5.4}, it follows that Nichols algebras of $V$ are finite dimensional if and only if $p=-1,~rm=-1$ or $rm=1,~p=\pm \xi$. i.e. Nichols algebras $\mathcal{B}(V)$ are 8 or 16 dimensional.
 \end{proof}
 \end{theorem}

\bibliographystyle{unsrt}

\end{document}